\documentclass[twoside]{article}

\usepackage[preprint]{aistats2026}

\usepackage[round]{natbib}

\usepackage{hyperref}
\usepackage{url}
\usepackage{amsmath,amsfonts,amssymb,amsthm,array}
\usepackage{algorithm}
\usepackage{algpseudocode}

\usepackage{tabularx}  

\usepackage{booktabs}
\usepackage{siunitx}
\usepackage{makecell}

\usepackage{enumitem}
\usepackage{microtype}
\usepackage{caption}
\usepackage{fancyhdr}
\usepackage{graphicx}
\usepackage{float}

\newtheorem{theorem}{Theorem}
\newtheorem{lemma}{Lemma}
\newtheorem{corollary}{Corollary}

\newtheorem{assumption}{Assumption}

\newcommand{\eqdef}{\; { := }\;}

\newcommand{\E}{\mathbb{E}}

\newcommand{\cD}{\mathcal{D}}

\newcommand{\R}{\mathbb{R}}

\begin{document}

\twocolumn[

\aistatstitle{Muon is Provably Faster with Momentum Variance Reduction}

\aistatsauthor{ Xun Qian  \And Hussein Rammal \And Dmitry Kovalev \And Peter Richt\'arik}

\aistatsaddress{ KAUST\footnote{Center of Excellence for Generative AI,  Thuwal, Saudi Arabia}   \And KAUST \And Yandex Research \And KAUST } ]

\begin{abstract}
 Recent empirical research has demonstrated that deep learning optimizers based on the linear minimization oracle (LMO) over specifically chosen Non-Euclidean norm balls, such as Muon and Scion, outperform Adam-type methods in the training of large language models. In this work, we show that such optimizers can be provably improved by replacing their vanilla momentum by momentum variance reduction (MVR). Instead of proposing and analyzing MVR variants of Muon and Scion separately, we incorporate MVR into the recently proposed Gluon framework, which captures Muon, Scion and other specific Non-Euclidean LMO-based methods as special cases, and at the same time works with a more general smoothness assumption which better captures the layer-wise structure of neural networks. In the non-convex case, we incorporate MVR into Gluon in three different ways. All of them improve the convergence rate from ${\cal O} (\frac{1}{K^{1/4}})$ to ${\cal O} (\frac{1}{K^{1/3}})$. Additionally, we provide improved rates in the star-convex case. Finally, we conduct several numerical experiments that verify the superior performance of our proposed algorithms in terms of iteration complexity.
\end{abstract}

\section{INTRODUCTION}

Since the success of ChatGPT, neural networks tend to grow larger and deeper. Hence an effective and stable trainer is very significant in modern machine learning. After the dominance of Adam-type methods (for instance, Adam \citep{kinga2015method}, AMSGrad \citep{reddi2019convergence}, AdamW \citep{loshchilovdecoupled}, and Adan \citep{xie2024adan}) over a decade, Muon \citep{jordan6muon} has shown the potential to outperform them in the training of large language models empirically. Later, the convergence of Muon was achieved in \citet{li2025note}, \citet{kovalev2025understanding}, and \citet{shen2025convergence}. Scion \citep{pethick2025training} improves upon Muon by applying the linear minimization oracle (LMO) based rule to all layers. Gluon \citep{riabinin2025gluon} generalizes Muon and Scion by exploring the layer-wise structure of neural networks, and numerical experiments therein close the gap between theory and practice of LMO-based optimizers for large language models (LLMs). Momentum variance reduction (MVR) is proposed in STORM by \citet{cutkosky2019momentum} for non-convex problems, where faster convergence rate than stochastic gradient descent (SGD) is obtained and no large batch size or checkpoint gradient is needed. In this work, our goal is to incorporate MVR into Muon-type methods to obtain faster convergence rate.

We consider the following optimization problem: 
\begin{equation}\label{p:p}
	\min_{X \in {\cal S}} \left\{  f(X) \eqdef \E_{\xi \sim \cD} [f_{\xi} (X) ] \right\}, 
\end{equation}
where ${\cal S}$ is a finite-dimensional vector space, $f_\xi: {\cal S} \to \R$ are potentially non-convex and non-smooth but continuously differentiable functions, and $\cD$ is the probability distribution of the training data. We assume that $f(\cdot)$ has a lower bound, i.e., $\inf_{X \in {\cal S}} > -\infty$ in this work. 

\paragraph{Related Work}
There are several works that adapt the momentum variance reduction into Muon-type methods in the reference, such as \citet{sfyraki2025lions} and \citet{chang2025convergence}. However, they did not consider the layer-wise structure of neural networks, and did not adopt the refined $(L^0, L^1)$-smoothness assumption \citep{zhang2019gradient}. The algorithm 4 in \citet{sfyraki2025lions} can achieve ${\cal O} \left(  \frac{1}{K^{1/3}}  \right)$ convergence rate, where $K$ denotes
the number of iterations, but the large batch size $K^{1/3}$ is needed at the initial step. The search direction in that algorithm is the convex combination of the stochastic gradient and the momentum $M^k$ in our proposed Gluon-MVR-1 algorithm, and the weigh decay is also adopted. The Algorithm 6 there, where an additional clipping is used, can also achieve ${\cal O} \left(  \frac{1}{K^{1/3}}  \right)$ convergence rate with $p=2$.  However, stronger assumptions are needed, i.e., the Assumption 3(b) (L-Lipschitz gradient of each data sample) and Assumption 5 (Bounded gradient norm over the set). 

The algorithm in \citet{chang2025convergence} that can achieve ${\tilde O}(\frac{1}{K^{1/3}})$ is only MVR2 with $\gamma=1$, where ${\tilde {\cal O}}$ hides some logarithmic factors. The search direction there is the same with that of Gluon-MVR-1, and the decreasing step size is used. While Gluon-MVR-1 has both constant and decreasing step-size versions. It needs the assumption that all data sample are Lipschitz smooth, which is strong, and it is not for the general norm in the linear minimization oracle. 

There is another concurrent MVR adaptation work, i.e., LiMuon in \citet{huang2025limuonlightfastmuon}. They also only consider the LMO with the spectral norm. The search direction in LiMuon with option 1 is the same with that of Gluon-MVR-1. LiMuon with option 2 is an inexact version of option 1 in the sense that a randomized Singular Value Decomposition is applied instead, and an additional assumption on the singular value of the momentum $M^k$ is used. The $(L_0-L_1)$-smoothness assumption is used with respect to the Euclidean norm, but not for the layer-wise case. 

Overall, the MVR2 with $\gamma=1$ in \citet{chang2025convergence} and LiMuon with option 1 in \citet{huang2025limuonlightfastmuon} is exactly Gluon-MVR-1 with the spectral norm and $p=1$. Moreover, from the numerical experiments therein, the performance of Gluon-MVR-1 is the best among the compared algorithms. While we have a comprehensive study of the MVR adaptation to Muon-type methods in three possible ways, which results in three algorithms---Gluon-MVR-1/2/3. All the three algorithms can obtain ${\cal O} \left(  \frac{1}{K^{1/3}}  \right)$ convergence rate. We study them for the LMO with general norm and under the layer-wise $(L^0, L^1)$-smoothness assumption. Furthermore, the numerical experiments show the superior performance of Gluon-MVR-2 and Gluon-MVR-3 over Gluon-MVR-1. This highlights the significance of our work in the MVR adaptation direction.

\paragraph{Contributions}
Our key contributions can be summarized as follows: \\
i), In the star-convex case, we adapt the momentum variance reduction in STORM \citep{cutkosky2019momentum} to Muon \citep{kovalev2025understanding}, and improve the convergence rate from ${\tilde {\cal O}} \left(  \frac{1}{K^{1/3}}  \right)$  to ${\tilde {\cal O}} \left(  \frac{1}{K^{1/2}}  \right)$ . \\
ii), In the non-convex case, we incorporate MVR into Gluon \citep{riabinin2025gluon} in three different ways, which results in three algorithms: Gluon-MVR-1, Gluon-MVR-2, and Gluon-MVR-3. We adopt the layer-wise $(L^0, L^1)$-smoothness assumption (Assumption \ref{as:L0L1smooth}) in Gluon to capture the layer-wise geometry of neural networks. All of these three methods achieve ${\cal O} \left(  \frac{1}{K^{1/3}}  \right)$ convergence rate, which is optimal \citep{arjevani2023lower}. We also study Gluon-MVR-1 with the decreasing step size, which allows us to use larger step size in the early stage of training, and ${ \tilde {\cal O}} \left(  \frac{1}{K^{1/3}}  \right)$ convergence rate is obtained within some logarithmic factors. \\
iii), We also provide convergence results for Gluon with constant step size. ${\cal O} \left(  \frac{1}{K^{1/4}}  \right)$ convergence rate is achieved and an explanation is given for its superior practical performance by the iteration complexity. \\
iv), We conduct numerical experiments by training NanoGPT with 124M parameters on FineWeb-10B dataset. Compared to Gluon, our algorithms converge faster and could achieve lower validation losses, which validates our theoretical predictions that MVR reduces stochastic noise in momentum estimates, enabling faster convergence.

\section{MUON-MVR IN THE STAR-CONVEX CASE}

In this section, we consider the star-convex case, i.e., $f$ satisfies the following assumption. We assume that there exists a solution $X^* \in {\cal S}$ to problem (\ref{p:p}) in this section. 
\begin{assumption}[Star-convexity]\label{eq:cvx}
We assume that the objective function $f(x)$ is star-convex, that is, the following inequality holds:
\begin{equation}
	f(\beta X^* + (1 - \beta)X) \leq \beta f(X^*) + (1-\beta)f(X), \nonumber 
\end{equation}
for all $X \in {\cal S}$ and $\beta \in [0, 1]$, where $X^* \in {\cal S}$ is a solution. 
\end{assumption}

{\bf Stochastic gradient estimator.}
We assume access to a stochastic estimator $\nabla f_{\xi} : {\cal S} \to {\cal S}$ of the gradient $\nabla f(\cdot)$, where $\xi \sim \cD$ is a random variable sampled from a probability distribution $\cD$. 
\begin{assumption}\label{eq:sigma}
We assume that the stochastic gradient estimator $\nabla f_{\xi}$ is unbiased and has bounded variance, that is, the following relations hold: $\E_{\xi \sim \cD} [\nabla f_{\xi} (X)] = \nabla f(X)$ and 
\begin{equation}
	\E_{\xi \sim \cD} [ \|\nabla f_{\xi} (X)- \nabla f(X)\|_2^2 ] \leq \sigma^2
	\quad\text{for all}\;
	X \in {\cal S}, \nonumber 
\end{equation}
where $\sigma \geq 0$ is a positive variance parameter, and $\|\cdot\|_2$ is the standard Euclidean norm induced by the inner product. 
\end{assumption}

{\bf Non-Euclidean norm setting and Lipschitz continuous gradient.}
We assume that vector space ${\cal S}$ is equipped with a norm $\| \cdot\| :  {\cal S} \to \R_+$, which possibly does not coincide with the Euclidean norm $\| \cdot\|_2$. 
\begin{assumption}\label{eq:L}
We assume that the stochastic gradient $\nabla f_{\xi}$ is mean-squared Lipschitz continuous with respect to the norm $\| \cdot\|$, that is, the following inequality holds:
\begin{equation}
	\E_{\xi\sim\cD} [ \| \nabla f_{\xi}(X) - \nabla f_{\xi} (Y)\|_\star^2 ]\leq L^2  \|X-Y\|^2, \nonumber 
\end{equation}
for all $	X, Y \in {\cal S}$, where $L > 0$ is the gradient Lipschitz constant, and $\| \cdot\|_\star : {\cal S} \to \R_+$ is the dual norm associated with $\| \cdot\|$, i.e., $\|X\|_\star = \sup_{\| Y\|\leq 1} \langle X, Y \rangle$ for all $X \in {\cal S}$.
\end{assumption}
In particular, this assumption implies that the gradient $\nabla f(\cdot)$ is $L$-Lipschitz with respect to the norm $\| \cdot\|$, that is, the following inequality holds:
\begin{equation}
	\| \nabla f(X) - \nabla f(Y) \|_\star \leq L \|X-Y\|
\end{equation}
for all $X, Y \in {\cal S}$. It is important to highlight that while Assumption \ref{eq:L} uses the dual norm $\| \cdot\|_\star$ to measure the difference between the gradients, the variance in Assumption \ref{eq:sigma} is measured with respect to the Euclidean norm $\| \cdot\|_2$, which is necessary to properly utilize the unbiasedness property of the stochastic gradient estimator $\nabla f_{\xi}$. Therefore, we need to provide a connection between these norms using the following assumption. 
\begin{assumption}\label{eq:rho}
We assume that $\|X\|_2 \leq  \|X\|_\star \leq \rho \cdot \|X\|_2$, for all $X \in {\cal S}$, where $\rho > 0$ is a positive constant.
\end{assumption}

The following is the Muon-MVR algorithm. We apply the momentum variance reduction technique  in STORM \citep{cutkosky2019momentum}. That is, we add the term $(1-\alpha) (\nabla f_{\xi^{k+1}} (X^{k+1}) - \nabla f_{\xi^{k+1}} (X^k))$ to the update of the momentum in Muon \citep{kovalev2025understanding}: $M^{k+1} = (1-\alpha) M^k + \alpha \nabla f_{\xi^{k+1}} (X^{k+1})$. Then we get the update of $M^k$ in (\ref{eq:m}) in Algorithm \ref{alg:muon-mvr}. 

\begin{algorithm}
	\caption{Muon-MVR}
	\label{alg:muon-mvr}
	\begin{algorithmic}[1]
		\State \textbf{Input:} $M^0 = \nabla f_{\xi^0} (X^0)$, where $\xi^0 \sim \cD$, $\alpha \in (0, 1)$, weight decay $\beta \in (0, 1)$
		\For{$k=0,1, ..., K-1$}
		\State Compute $X^{k+1}$ over $\mathcal{B}^k := \{X \in \mathcal{S} : \|X - X^k\| \leq  \eta\}$: as follows:
		\begin{align}\label{eq:x}
			X^{k+1} = \arg\min_{X\in\mathcal{B}^k } \langle M^k,X-X^k \rangle 
		\end{align}
		\State Let $X^{k+1} = (1-\beta)X^{k+1}$, sample $\xi^{k+1} \sim \cD$ and compute $M^{k+1}$ as follows:
		\begin{align}
			M^{k+1} =& (1-\alpha) (M^k - \nabla f_{\xi^{k+1}} (X^k) ) \nonumber  \\ 
			& + \nabla f_{\xi^{k+1}} (X^{k+1}) \label{eq:m}
		\end{align}
		\EndFor
	\end{algorithmic}
\end{algorithm}

We have the following convergence results for Muon-MVR. 

\begin{theorem}\label{th:muon-mvr}
	Let Assumptions \ref{eq:cvx}, \ref{eq:sigma}, \ref{eq:L}, and \ref{eq:rho} hold, $\alpha \geq \beta$, and $\eta \geq \beta\max\{\|X^0\|,\|X^*\| \}$. Then, the following inequality holds for Algorithm \ref{alg:muon-mvr}:
	\begin{eqnarray*}
		\E[f(X^K) - f(X^*)]
		&\leq& (1-\beta)^K \left(  f(X^0) - f(X^*)  \right) \\ 
		&&  +  \frac{4\eta^2 L}{\beta}  + \frac{4\sqrt{\alpha}\eta \rho \sigma}{\beta}  + \frac{8\rho \eta^2 L}{\sqrt{\alpha} \beta} \\ 
		&& + K (1-\beta)^{K-1} 2\eta \rho \sigma. 
	\end{eqnarray*}
\end{theorem}

\begin{corollary}\label{co:muon-mvr}
	Under the consitions of Theorem \ref{th:muon-mvr}, assume $\|X^0\| \leq \|X^*\| = D$. Let $\eta = \beta D$, and $\alpha = \beta = \frac{2 \ln K}{K}$. Then we have 
	\begin{align*}
	\E \left[  f(X^K) - f(X^*)  \right] &\leq \frac{f(X^0) - f(X^*)}{K^2}  +  \frac{4\rho \sigma D \ln K}{(1-\beta)K^2} \\ 
	&  +  \frac{8LD^2 \ln K}{K}  + \frac{4\sqrt{2}\rho \sigma D \sqrt{\ln K}}{\sqrt{K}} \\ 
	& +  \frac{8\sqrt{2}\rho LD^2\sqrt{\ln K}}{\sqrt{K}}. 
	\end{align*}
\end{corollary}

The obtained ${\tilde {\cal O}} \left(  \frac{1}{K^{1/2}}  \right)$ convergence rate of Muon-MVR in Corollary \ref{co:muon-mvr} is faster than that of Muon in the star-convex case, which is ${\tilde {\cal O}} \left(  \frac{1}{K^{1/3}}  \right)$ \citep{kovalev2025understanding}.

\section{GLUON-MVR-1}

In this and following sections, we consider the non-convex and layer-wise case, where $X = [X_1, ..., X_p]$ with $X_i \in {\cal S}_i$. First, we introduce the assumptions. We adopt the following layer-wise $(L^0, L^1)$-smoothness assumption used in Gluon \citep{riabinin2025gluon} to enable us to explore the layer-wise geometry of neural networks in the training. 

\begin{assumption}\label{as:L0L1smooth}
The function $f: \mathcal{S} \mapsto \mathbb{R}$ is layer-wise $(L^0, L^1)$-smooth with constants $L^0 := (L^0_1, \ldots, L^0_p) \in \mathbb{R}_+^p$ and $L^1 := (L^1_1, \ldots, L^1_p) \in \mathbb{R}_+^p$. That is, the inequality
\begin{eqnarray*}
	&&\quad  \|\nabla_i f(X) - \nabla_i f(Y)\|_{(i)\star} \\ 
	&& \leq \left(L^0_i + L^1_i \|\nabla_i f(X)\|_{(i)\star}\right) \|X_i - Y_i\|_{(i)}
	\label{eq:layer_smooth}
\end{eqnarray*}
holds for all $i = 1, \ldots, p$ and all $X = [X_1, \ldots, X_p] \in \mathcal{S}$, $Y = [Y_1, \ldots, Y_p] \in \mathcal{S}$, where $\|\cdot\|_{(i)\star}$ is the dual norm associated with $\|\cdot\|_{(i)}$ (i.e., $\|X_i\|_{(i)\star} := \sup_{\|Z_i\|_{(i)} \leq 1} \langle X_i, Z_i \rangle_{(i)}$ for any $X_i \in \mathcal{S}_i$).
\end{assumption}

\begin{assumption}\label{as:boundedvariance}
	The stochastic gradient estimator $\nabla f_\xi : {\cal S} \to {\cal S}$ is unbiased and has bounded variance. That is, $\E_{\xi \sim {\cal D}} \left[  \nabla f_\xi (X)  \right] = \nabla f(X)$ for all $X \in {\cal S}$ and there exists $\sigma \geq 0$ such that
	$$
	\E_{\xi \sim {\cal D}} \left[  \|\nabla_i f_\xi (X) - \nabla_i f(X) \|_2^2 \right] \leq \sigma^2, 
	$$
	for any $X \in {\cal S}$ and $ i = 1, ..., p$. 
\end{assumption}

\begin{assumption}\label{as:rho}
	$\| X\|_{(i)\star} \leq \rho\|X\|_2$ for any $X$ and $i \in \{1, ..., p\}$. 
\end{assumption}

Noticed that $\rho$ in Assumtion \ref{as:rho} always exists because of the norm equivalence in finite dimensional inner product space. For instance, if $X$ is a $m\times n$ matrix with rank $r$, and the spectral norm is adopted, then $\rho = \sqrt{r}$.

\begin{assumption}\label{as:HV}
	There exists $\delta_i \geq 0$ for all $i \in \{  1, ..., p  \}$ such that 
	\begin{align}
		& \E [ \| \nabla_i f_{\xi} (X)  - \nabla_i f_{\xi} (Y) - (\nabla_i f(X) - \nabla_i f(Y))\|_2^2 ] \nonumber \\ 
		\leq & \delta_i^2 \|X_i-Y_i\|_{(i)}^2, \label{eq:HVdelta}
	\end{align}
	where $\| \cdot\|_2$ is the Euclidean norm. 
\end{assumption}

Assumption \ref{as:HV} characterizes the Hessian variance of the objective function, and $\delta_i$ can be upper-bounded by the expected Hessian variance. $\delta_i$ can also be upper-bounded by the expected smoothness constant since the left-hand size of (\ref{eq:HVdelta}) can actually be upper-bounded by $\E [  \| \nabla_i f_{\xi} (X)  - \nabla_i f_{\xi} (Y)\|_2^2  ]$. 

For the non-convex case, we apply the momentum variance reduction technique in STORM \citep{cutkosky2019momentum} by three different ways. The first one is Algorithm \ref{alg:gluon_cs_mvr}, which we call Gluon-MVR-1. The way of adapting MVR in Gluon-MVR-1 is the same as the way in Muon-MVR. That is, we change the momentum $M_i^k$ in Gluon \citep{riabinin2025gluon} to $M_i^k = \nabla_i f_{\xi^k} (X^k)  + \beta (M_i^{k-1} - \nabla_i f_{\xi^k} (X^{k-1}) )$, and use the constant step size. 

\begin{algorithm}
	\caption{Gluon-MVR-1}\label{alg:gluon_cs_mvr}
	\begin{algorithmic}[1]
		\State \textbf{Input:} Initial model parameters $X^0 = [X_1^0, \dots, X_p^0] \in \mathcal{S}$, momentum $M^0 = [M_1^0, \dots, M_p^0] \in \mathcal{S}$, momentum decay factors $\beta \in [0, 1)$ for all iterations $k \geq 0$
		\For{$k = 0, 1, 2, \dots, K - 1$}
		\State Sample $\xi^k \sim \mathcal{D}$
		\For{$i = 1, 2, \dots, p$}
		\State Compute stochastic gradient $\nabla_i f_{\xi^k}(X^k)$ 
		\State Update momentum $M_i^k = \nabla_i f_{\xi^k} (X^k)  + \beta (M_i^{k-1} - \nabla_i f_{\xi^k} (X^{k-1}) )$ for layer $i$
		\State Choose adaptive stepsize/radius $t_i \eta > 0$ 
		\State Update parameters for layer $i$ via LMO over $\mathcal{B}_i^k := \{X_i \in \mathcal{S}_i : \|X_i - X_i^k\|_{(i)} \leq t_i \eta\}$:
		\begin{equation}
			X_i^{k+1} = \text{LMO}_{\mathcal{B}_i^k}(M_i^k) := \arg \min_{X_i \in \mathcal{B}_i^k} \langle M_i^k, X_i \rangle_{(i)}
		\end{equation}
		\EndFor
		\State Update full parameter vector $X^{k+1} = [X_1^{k+1}, \dots, X_p^{k+1}]$
		\EndFor
	\end{algorithmic}
\end{algorithm}

The convergence results are stated in Theorem \ref{th:cs-Gluon-mvr} and Corollary \ref{co:Gluon-mvr-1}, where ${\cal O} \left(  \frac{1}{K^{1/3}}  \right)$ convergence rate can be achieved. 

\begin{theorem}\label{th:cs-Gluon-mvr}
	Let Assumptions \ref{as:L0L1smooth}, \ref{as:boundedvariance}, \ref{as:rho}, and Assumption \ref{as:HV} hold, and fix $\epsilon>0$. Let $X^0, ..., X^{K-1}$ be the iterates of Algorithm \ref{alg:gluon_cs_mvr}, and $M_i^0 = \nabla_i f_{\xi^0} (X^0)$. Denote $\alpha \eqdef 1-\beta$. 
	
	1. If $L_i^1 =0$, then
	\begin{eqnarray*}
		&& \min_{k=0, ..., K-1} \sum_{i=1}^p t_i \E \left[  \|\nabla_i f(X^k) \|_{(i)\star} \right] \\ 
		&\leq&  \frac{\Delta^0}{\eta K} + \frac{2\sum_{i=1}^p t_i \rho \sigma}{\alpha K} + 2\sqrt{\alpha} \sum_{i=1}^p t_i \rho \sigma  \\ 
		&& + \frac{2\sum_{i=1}^p \delta_i t_i^2 \rho  \eta}{\sqrt{\alpha}} + \frac{\sum_{i=1}^p L_i^0 t_i^2 \eta}{2}, 
	\end{eqnarray*}
	where $\Delta^0 \eqdef f(X^0) - \inf_{X \in {\cal S}} f(X)$. By choosing $\eta = K^{-2/3}$ and $\alpha = K^{-2/3}$, we have 
	$ \min_{k=0, ..., K-1} \sum_{i=1}^p t_i \E \left[  \|\nabla_i f(X^k) \|_{(i)\star} \right]  \leq {\cal O} \left(  \frac{1}{K^{1/3}}  \right)$. 
	
	2. If $L_i^1 \neq 0$, we let $\eta \leq \min_i \frac{1}{L_i^1 t_i}$. Then 
	\begin{eqnarray*}
		&& \min_{k=0, ..., K-1} \sum_{i=1}^p t_i \E \left[  \|\nabla_i f(X^k) \|_{(i)\star} \right] \\ 
		&\leq&  \frac{2\Delta^0}{\eta K} + \frac{4\sum_{i=1}^p t_i \rho \sigma}{\alpha K} + 4\sqrt{\alpha} \sum_{i=1}^p t_i \rho \sigma \\ 
		&& + \frac{4\sum_{i=1}^p \delta_i t_i^2 \rho \eta}{\sqrt{\alpha}} + {\sum_{i=1}^p L_i^0 t_i^2 \eta}. 
	\end{eqnarray*}
	By choosing the parameters as $\eta = \min \left(   \min_i \frac{1}{L_i^1 t_i},  K^{-2/3}  \right)$ and $\alpha = K^{-2/3}$, we can obtain that $\min_{k=0, ..., K-1} \sum_{i=1}^p t_i \E \left[  \|\nabla_i f(X^k) \|_{(i)\star} \right]  \leq {\cal O}\left(  \frac{1}{K^{1/3}} + \frac{\max_{i} L_i^1 t_i}{K}  \right). $

\end{theorem}

From Theorem \ref{th:cs-Gluon-mvr}, we can get the iteration complexity in the following corollary. The proof is omitted due to its simplicity.
\begin{corollary}\label{co:Gluon-mvr-1}
Under the conditions of Theorem \ref{th:cs-Gluon-mvr}, to reach the precision $$ \min_{k=0, ..., K-1} \sum_{i=1}^p t_i \E \left[  \|\nabla_i f(X^k) \|_{(i)\star} \right]  \leq \epsilon,$$ it is sufficient to choose the parameters as follows: 
	\begin{align}
	\eta = & {\cal O} \left(  \min \left\{  \frac{\epsilon}{\sum_{i=1}^p L_i^0 t_i^2}, \frac{\epsilon}{\sum_{i=1}^p \delta_i t_i^2 \rho}, \right. \right.  \nonumber \\   
	& \left.  \left.  \hskip 15mm  \frac{\epsilon^2}{(\sum_{i=1}^p \delta_i t_i^2) (\sum_{i=1}^p t_i) \rho \sigma }  \right\}  \right), \\ 
	\alpha = &  {\cal O} \left(  \min \left\{  1, \frac{\epsilon^2}{(\sum_{i=1}^p t_i)^2 \rho^2 \sigma^2}  \right\}  \right), \\
	K = & {\cal O} \left(  \max \left\{  \frac{\sum_{i=1}^p t_i \rho \sigma}{\epsilon}, \frac{(\sum_{i=1}^p t_i)^3 \rho^3 \sigma^3}{\epsilon^3},  \right. \right. \nonumber \\ 
	&  \left.  \left.  \hskip 15mm  \frac{\sum_{i=1}^p L_i^0 t_i^2 \Delta^0}{\epsilon^2},  \frac{\sum_{i=1}^p \delta_i t_i^2 \Delta^0 \rho }{\epsilon^2}, \right. \right. \nonumber \\ 
	& \left. \left.  \hskip 15mm \frac{(\sum_i \delta_i t_i^2) (\sum_i t_i) \Delta^0 \rho  \sigma}{\epsilon^3}  \right\}  \right), 
\end{align}
for $L_i^1 = 0$; and 
	\begin{align*}
	& \eta = {\cal O} \left(  \min \left\{  \frac{\epsilon}{\sum_{i=1}^p L_i^0 t_i^2}, \frac{\epsilon}{\sum_{i=1}^p \delta_i t_i^2 \rho }, \right. \right. \nonumber \\ 
	& \left. \left. \hskip 12mm  \frac{\epsilon^2}{(\sum_{i=1}^p \delta_i t_i^2) (\sum_{i=1}^p t_i) \rho \sigma },  \frac{1}{\max_i L_i^1 t_i} \right\}  \right), \\ 
	&  \alpha = {\cal O} \left(  \min \left\{  1, \frac{\epsilon^2}{(\sum_{i=1}^p t_i)^2 \rho^2 \sigma^2}  \right\}  \right), \\
	& K = {\cal O} \left(  \max \left\{  \frac{\sum_{i=1}^p t_i \rho \sigma}{\epsilon}, \frac{(\sum_{i=1}^p t_i)^3 \rho^3 \sigma^3}{\epsilon^3},  \right.  \right. \nonumber \\ 
	& \left. \left.  \hskip 12mm   \frac{\sum_{i=1}^p L_i^0 t_i^2 \Delta^0}{\epsilon^2}, \frac{\max_i L_i^1 t_i \Delta^0}{\epsilon}, \right. \right. \nonumber  \\ 
	& \hskip 8mm   \left. 	\left.  \frac{\sum_{i=1}^p \delta_i t_i^2 \Delta^0 \rho}{\epsilon^2},  \frac{(\sum_i \delta_i t_i^2) (\sum_i t_i) \Delta^0 \rho \sigma}{\epsilon^3}   \right\}  \right), 
\end{align*}
for $L_i^1 \neq 0$. 
\end{corollary}

\paragraph{Gluon with Constant Step Size}

We also provide the analysis for Gluon with constant step size in the Appendix. Even though the constant step size is tiny compared to the decreasing step size in Gluon \citep{riabinin2025gluon}, there are no logarithmic factors in the convergence results. In particular, when $L_i^1 \neq 0$,  the iteration complexity is 
\begin{align*}
K &= {\cal O} \left(  \max \left\{  \frac{\sum_{i=1}^p t_i \rho \sigma}{\epsilon}, \frac{(\sum_{i=1}^p t_i)^3 \rho^3 \sigma^3}{\epsilon^3},  \right. \right. \\ 
& \left. \left.   \frac{\sum_{i=1}^p L_i^0 t_i^2 \Delta^0}{\epsilon^2}, \frac{\max_i L_i^1 t_i \Delta^0}{\epsilon},    \frac{(\sum_{i=1}^p t_i)^2 \Delta^0 \rho^2 \sigma^2}{\max_i L_i^1 t_i \epsilon^3},  \right. \right.   \\ 
&   \left. 	\left.   \frac{(\sum_{i=1}^p L_i^0t_i^2) (\sum_{i=1}^p t_i)^2 \Delta^0 \rho^2 \sigma^2}{\epsilon^4}  \right\}  \right). 
\end{align*}
It was shown in \citet{riabinin2025gluon} that $L_i^0$ is usually close to zero in practice. Then the above iteration complexity would approach ${\cal O} \left(  \frac{1}{\epsilon^3}  \right)$. This may explain the superior performance of the Muon-type method in practice.

\section{GLUON-MVR-1 WITH DECREASING STEP SIZE}

In this section, we consider Gluon-MVR-1 with decreasing step size, which allows us to use larger step size in the early stage of training. First,we need two technical lemmas. 

\begin{algorithm}
	\caption{Gluon-MVR-1 with decreasing step size}\label{alg:ds_gluon_mvr}
	\begin{algorithmic}[1]
		\State \textbf{Input:} Initial model parameters $X^0 = [X_1^0, \dots, X_p^0] \in \mathcal{S}$, momentum $M^0 = [M_1^0, \dots, M_p^0] \in \mathcal{S}$, momentum decay factors $\beta^k \in [0, 1)$ for all iterations $k \geq 0$
		\For{$k = 0, 1, 2, \dots, K - 1$}
		\State Sample $\xi^k \sim \mathcal{D}$
		\For{$i = 1, 2, \dots, p$}
		\State Compute stochastic gradient $\nabla_i f_{\xi^k}(X^k)$ 
		\State Update momentum $M_i^k = \nabla_i f_{\xi^k} (X^k)  + \beta^k (M_i^{k-1} - \nabla_i f_{\xi^k} (X^{k-1}) )$ for layer $i$
		\State Choose adaptive stepsize/radius $t_i^k > 0$ 
		\State Update parameters for layer $i$ via LMO over $\mathcal{B}_i^k := \{X_i \in \mathcal{S}_i : \|X_i - X_i^k\|_{(i)} \leq t_i^k\}$:
		\begin{equation}
			X_i^{k+1} = \text{LMO}_{\mathcal{B}_i^k}(M_i^k) := \arg \min_{X_i \in \mathcal{B}_i^k} \langle M_i^k, X_i \rangle_{(i)}
		\end{equation}
		\EndFor
		\State Update full parameter vector $X^{k+1} = [X_1^{k+1}, \dots, X_p^{k+1}]$
		\EndFor
	\end{algorithmic}
\end{algorithm}

\begin{lemma}\label{lm:lm1ds}
	Let $\alpha^k = (k+1)^{-2/3}$, $\beta^k = 1 - (k+1)^{-2/3}$, and denote $\beta^{a:b} \eqdef \prod_{k=a}^b \beta^k$. Then for all $K \in \mathbb{N}_{\geq 1}$, we have $\sum_{k=0}^{K-1} (k+1)^{-\frac{2}{3}}  \sqrt{ \sum_{\tau =0}^k (\beta^{(\tau+1):k} \alpha^\tau)^2 } \leq (12+\sqrt{2e^3} \log K)$. 
\end{lemma}

\begin{lemma}\label{lm:lm2ds}
	Let $t_i^k = t_i (k+1)^{-2/3}$ with $t_i>0$, $\beta^k = 1 - (k+1)^{-2/3}$, and denote $\beta^{a:b} \eqdef \prod_{k=a}^b \beta^k$. Then for all $K \in \mathbb{N}_{\geq 1}$, we have $\sum_{k=0}^{K-1} (k+1)^{-\frac{2}{3}}  \sqrt{ \sum_{\tau =1}^k (\beta^{(\tau+1):k} t_i^\tau)^2 } \leq t_i(12+\sqrt{2e^3} \log K)$. 
\end{lemma}

Then we can get ${ \tilde {\cal O}} \left(  \frac{1}{K^{1/3}}  \right)$ convergence rate within some logarithmic factors in Theorem \ref{th:ds-Gluon-mvr}. 
\begin{theorem}\label{th:ds-Gluon-mvr}
	Let Assumptions \ref{as:L0L1smooth}, \ref{as:boundedvariance}, \ref{as:rho}, and Assumption \ref{as:HV} hold. Let $X^0, ..., X^{K-1}$ be the iterates of Algorithm \ref{alg:ds_gluon_mvr} run with $\beta^k = 1 - (k+1)^{-2/3}$, $t_i^k = t_i(k+1)^{-2/3}$ for some $t_i>0$, and $M_i^0 = \nabla_i f_{\xi^0} (X^0)$. 
	
	1. If $L_i^1 =0$, then
	\begin{eqnarray*}
		&&	\min_{k=0, ..., K-1} \sum_{i=1}^p t_i \E \left[  \|\nabla_i f(X^k) \|_{(i)\star}  \right] \\ 
		&\leq& \frac{\Delta^0}{K^{1/3}} + \frac{1}{K^{1/3}}  \sum_{i=1}^p  \left[  \rho \sigma t_i (24 + 2\sqrt{2e^3} \log K)  \right. \\ 
		&&  \left.  \hskip 10mm + \rho \delta_i t_i^2 (24 + 2\sqrt{2e^3} \log K) + 2L_i^0 t_i^2  \right], 
	\end{eqnarray*}
	
	2. If $L_i^1 \neq 0$, then for $t_i = \frac{1}{L_i^1}$,  we have 
	\begin{eqnarray*}
		&&	\min_{k=0, ..., K-1} \sum_{i=1}^p t_i \E \left[  \|\nabla_i f(X^k) \|_{(i)\star}  \right] \\ 
		&\leq& \frac{\Delta^0}{K^{1/3}} + \frac{1}{K^{1/3}}  \sum_{i=1}^p  \left[  \frac{\rho \sigma}{L_i^1} (24 + 2\sqrt{2e^3} \log K)  \right. \\ 
		&&  \left. \hskip 9mm  + \frac{\rho \delta_i}{(L_i^1)^2}  (24 + 2\sqrt{2e^3} \log K) + \frac{2L_i^0}{(L_i^1)^2}  \right]. 
	\end{eqnarray*}

\end{theorem}

\section{GLUON-MVR-2}

In this section, we consider the second way of MVR adaptation. Recall that the momentum $M_i^k$ in Gluon is $M_i^k = \beta M_i^{k-1} +  \nabla_i f_{\xi^k} (X^k)$. Unlike replacing this $M_i^k$ with the MVR estimator \citep{cutkosky2019momentum} as in Gluon-MVR-1, we replace the stochastic gradient  $\nabla_i f_{\xi^k} (X^k)$ by the MVR estimator. This is also equivalent to applying an additional momentum to the MVR estimator. That is, $M_i^k = \beta M_i^{k-1} + (1-\beta) g_i^k$ with the MVR estimator 
\begin{equation}\label{eq:mvrest}
	g_i^k = \nabla_i f_{\xi^k} (X^k)  + (1-q) (g_i^{k-1} - \nabla_i f_{\xi^k} (X^{k-1}) ), 
\end{equation}
 for each layer. This results in the following algorithm: Gluon-MVR-2. 

\begin{algorithm}
	\caption{Gluon-MVR-2}\label{alg:gluon_cs_mvr-2}
	\begin{algorithmic}[1]
		\State \textbf{Input:} Initial model parameters $X^0 = [X_1^0, \dots, X_p^0] \in \mathcal{S}$, momentum $M^0 = [M_1^0, \dots, M_p^0] \in \mathcal{S}$, momentum decay factors $\beta \in (0, 1)$ and MVR parameter $q \in (0, 1]$ for all iterations $k \geq 0$
		\For{$k = 0, 1, 2, \dots, K - 1$}
		\State Sample $\xi^k \sim \mathcal{D}$
		\For{$i = 1, 2, \dots, p$}
		\State Compute stochastic gradient $\nabla_i f_{\xi^k}(X^k)$ 
		\State Update the MVR estimator $g_i^k = \nabla_i f_{\xi^k} (X^k)  + (1-q) (g_i^{k-1} - \nabla_i f_{\xi^k} (X^{k-1}) )$ for layer $i$
		\State Update momentum $M_i^k = \beta M_i^{k-1} + (1-\beta) g_i^k$  for layer $i$
		\State Choose adaptive stepsize/radius $t_i \eta > 0$ 
		\State Update parameters for layer $i$ via LMO over $\mathcal{B}_i^k := \{X_i \in \mathcal{S}_i : \|X_i - X_i^k\|_{(i)} \leq t_i \eta\}$:
		\begin{equation}\label{eq:Gluon-mvr-2-update}
			X_i^{k+1} = \text{LMO}_{\mathcal{B}_i^k}(M_i^k) := \arg \min_{X_i \in \mathcal{B}_i^k} \langle M_i^k, X_i \rangle_{(i)}
		\end{equation}
		\EndFor
		\State Update full parameter vector $X^{k+1} = [X_1^{k+1}, \dots, X_p^{k+1}]$
		\EndFor
	\end{algorithmic}
\end{algorithm}

The following theorem shows the convergence results of Gluon-MVR-2 under different parameter settings. We show how to choose parameters to obtain fast convergence rate in subsection \ref{subsec:cp-Gluon-MVR-2}. 

\begin{theorem}\label{th:cs-Gluon-mvr-2}
	Let Assumptions \ref{as:L0L1smooth}, \ref{as:boundedvariance}, \ref{as:rho}, and Assumption \ref{as:HV} hold. Let $X^0, ..., X^{K-1}$ be the iterates of Algorithm \ref{alg:gluon_cs_mvr-2}, and $M_i^0 = g_i^0 = \nabla_i f_{\xi^0} (X^0)$. Denote $\alpha \eqdef 1-\beta$ and $v \eqdef \frac{1-q}{1-\alpha}$. 
	
	1. If $L_i^1 =0$, then for $q < \alpha$, we have 
	\begin{eqnarray*}
		&& \min_{k=0, ..., K-1} \sum_{i=1}^p t_i \E \left[  \|\nabla_i f(X^k) \|_{(i)\star} \right] \\ 
		&\leq&   \frac{\Delta^0}{\eta K} + \frac{2\sum_{i=1}^p t_i \rho \sigma}{\alpha K} +       \frac{4\sqrt{q \alpha} \sum_{i=1}^p t_i  \rho \sigma}{\sqrt{(2-\alpha)(\alpha+\beta q)}}  \\ 
		&&  +  \frac{4\sqrt{\alpha} (1-q) \sum_{i=1}^p\delta_i t_i^2 \rho \eta}{\sqrt{(2-\alpha)(\alpha+\beta q)q}}   + \frac{2\sum_{i=1}^p L_i^0t_i^2 \eta}{\alpha} \\ 
		&&   +    \frac{4\sqrt{2\alpha} \sum_{i=1}^p t_i  \rho \sigma}{q K \sqrt{v+\alpha-1}}      + \frac{\sum_{i=1}^p L_i^0 t_i^2 \eta}{2}; 
	\end{eqnarray*}
	for $q\geq \alpha$, 
	\begin{eqnarray*}
		&& \min_{k=0, ..., K-1} \sum_{i=1}^p t_i \E \left[  \|\nabla_i f(X^k) \|_{(i)\star} \right] \\ 
		&\leq&  \frac{\Delta^0}{\eta K} + \frac{8\sum_{i=1}^p t_i \rho \sigma}{\alpha K} +       \frac{4\sqrt{q \alpha} \sum_{i=1}^p t_i  \rho \sigma}{\sqrt{(2-\alpha)(\alpha+\beta q)}} \\ 
		&&   +  \frac{4\sqrt{\alpha} (1-q) \sum_{i=1}^p\delta_i t_i^2 \rho \eta}{\sqrt{(2-\alpha)(\alpha+\beta q)q}}   \\ 
		&&  + \frac{2\sum_{i=1}^p L_i^0t_i^2 \eta}{\alpha} + \frac{\sum_{i=1}^p L_i^0 t_i^2 \eta}{2}. 
	\end{eqnarray*}
	
	2. If $L_i^1 \neq 0$, we let $\frac{\eta}{\alpha} \leq \min_{i}  \frac{1}{5L_i^1 t_i}$. Then we have $\left(  \frac{2}{\alpha} + \frac{1}{2}  \right) L_i^1 t_i \eta \leq \frac{1}{2}$ for all $i$, and for $q<\alpha$, 
	\begin{eqnarray*}
		&& \min_{k=0, ..., K-1} \sum_{i=1}^p t_i \E \left[  \|\nabla_i f(X^\tau) \|_{(i)\star} \right] \\ 
		&\leq&  \frac{2 \Delta^0}{\eta K} + \frac{4\sum_{i=1}^p t_i \rho \sigma}{\alpha K} +       \frac{8\sqrt{q \alpha} \sum_{i=1}^p t_i  \rho \sigma}{\sqrt{(2-\alpha)(\alpha+\beta q)}}  \\ 
		&&  +  \frac{8\sqrt{\alpha} (1-q) \sum_{i=1}^p\delta_i t_i^2 \rho \eta}{\sqrt{(2-\alpha)(\alpha+\beta q)q}}  \\ 
		&&   +    \frac{8\sqrt{2\alpha} \sum_{i=1}^p t_i  \rho \sigma}{q K \sqrt{v+\alpha-1}}   + \frac{4\sum_{i=1}^p L_i^0t_i^2 \eta}{\alpha} + {\sum_{i=1}^p L_i^0 t_i^2 \eta}; 
	\end{eqnarray*}
	for $q\geq \alpha$, we have 
	\begin{eqnarray*}
		&& \min_{k=0, ..., K-1} \sum_{i=1}^p t_i \E \left[  \|\nabla_i f(X^k) \|_{(i)\star} \right] \\ 
		&\leq&  \frac{2 \Delta^0}{\eta K} + \frac{16\sum_{i=1}^p t_i \rho \sigma}{\alpha K} +       \frac{8\sqrt{q \alpha} \sum_{i=1}^p t_i  \rho \sigma}{\sqrt{(2-\alpha)(\alpha+\beta q)}}  \\ 
		&&  +  \frac{8\sqrt{\alpha} (1-q) \sum_{i=1}^p\delta_i t_i^2 \rho \eta}{\sqrt{(2-\alpha)(\alpha+\beta q)q}}  \\ 
		&&   + \frac{4\sum_{i=1}^p L_i^0t_i^2 \eta}{\alpha} + {\sum_{i=1}^p L_i^0 t_i^2 \eta}. 
	\end{eqnarray*}
	
\end{theorem}

\subsection{Choices of Parameters for Gluon-MVR-2}\label{subsec:cp-Gluon-MVR-2}

In this subsection, we analyze these parameter settings in Theorem \ref{th:cs-Gluon-mvr-2}, and show how to obtain fast convergence rate. We consider the case where $L_i^0 = 0$ first. From Theorem \ref{th:cs-Gluon-mvr-2}, for $q\geq \alpha$, the upper-bound becomes 
\begin{eqnarray*}
&& \min_{k=0, ..., K-1} \sum_{i=1}^p t_i \E \left[  \|\nabla_i f(X^\tau) \|_{(i)\star} \right] \\ 
&\leq& {\cal O} \left(   \frac{1}{\eta K} + \frac{1}{\alpha K}  + \sqrt{\alpha}  + \frac{\sqrt{\alpha} \eta}{q} + \frac{\eta}{\alpha} \right), 
\end{eqnarray*}
where $ \frac{1}{\eta K} + \frac{\eta}{\alpha} + \sqrt{\alpha} \geq \frac{2\sqrt{2}}{K^{1/4}}$. Thus we can get at most ${\cal O} \left( \frac{1}{K^{1/4}} \right)$ convergence rate. For $q<\alpha$, we can get  
\begin{eqnarray*}
	&& \min_{k=0, ..., K-1} \sum_{i=1}^p t_i \E \left[  \|\nabla_i f(X^\tau) \|_{(i)\star} \right] \\ 
	&\leq& {\cal O} \left(   \frac{1}{\eta K} + \frac{1}{\alpha K}  + \sqrt{q}  + \frac{\eta}{\sqrt{q}} + \frac{\eta}{\alpha}  + \frac{1}{qK} \right). 
\end{eqnarray*}
By choosing $\eta = q = \frac{1}{K^{2/3}}$ and $\alpha \geq \frac{1}{K^{1/3}}$, we can achieve ${\cal O} \left(  \frac{1}{K^{1/3}}  \right)$ convergence rate. 

For the case where $L_i^1 \neq 0$, we can also get the above results as long as we choose $\eta$ and $\alpha$ such that $\frac{\eta}{\alpha} \leq \min_{i}  \frac{1}{5L_i^1 t_i}$ is satisfied, which generally holds for large $K$.

\section{GLUON-MVR-3}

In this section, we introduce the last way of MVR adaptation to Gluon, which is actually the combination of the first two ways. Adding the term $\beta (  \nabla_i f_{\xi^k} (X^k) -  \nabla_i f_{\xi^k} (X^{k-1}) )$ to the $M_i^k$ in Gluon-MVR-2 results in the update of $M_i^k$ in Gluon-MVR-3. That is, $M_i^k = \beta M_i^{k-1} + (1-\beta) g_i^k + \beta (  \nabla_i f_{\xi^k} (X^k) -  \nabla_i f_{\xi^k} (X^{k-1}) )$ with the MVR estimator $g_i^k$ in (\ref{eq:mvrest}). 

\begin{algorithm}
	\caption{Gluon-MVR-3}\label{alg:gluon_cs_mvr-3}
	\begin{algorithmic}[1]
		\State \textbf{Input:} Initial model parameters $X^0 = [X_1^0, \dots, X_p^0] \in \mathcal{S}$, momentum $M^0 = [M_1^0, \dots, M_p^0] \in \mathcal{S}$, momentum decay factors $\beta \in (0, 1)$ and MVR parameter $q \in (0, 1]$ for all iterations $k \geq 0$
		\For{$k = 0, 1, 2, \dots, K - 1$}
		\State Sample $\xi^k \sim \mathcal{D}$
		\For{$i = 1, 2, \dots, p$}
		\State Compute stochastic gradient $\nabla_i f_{\xi^k}(X^k)$ 
		\State Update the MVR estimator $g_i^k = \nabla_i f_{\xi^k} (X^k)  + (1-q) (g_i^{k-1} - \nabla_i f_{\xi^k} (X^{k-1}) )$ for layer $i$
		\State Update momentum $M_i^k = \beta M_i^{k-1} + (1-\beta) g_i^k + \beta (  \nabla_i f_{\xi^k} (X^k) -  \nabla_i f_{\xi^k} (X^{k-1}) )$ for layer $i$
		\State Choose adaptive stepsize/radius $t_i \eta > 0$ 
		\State Update parameters for layer $i$ via LMO over $\mathcal{B}_i^k := \{X_i \in \mathcal{S}_i : \|X_i - X_i^k\|_{(i)} \leq t_i \eta\}$:
		\begin{equation}\label{eq:Gluon-mvr-3-update}
			X_i^{k+1} = \text{LMO}_{\mathcal{B}_i^k}(M_i^k) := \arg \min_{X_i \in \mathcal{B}_i^k} \langle M_i^k, X_i \rangle_{(i)}
		\end{equation}
		\EndFor
		\State Update full parameter vector $X^{k+1} = [X_1^{k+1}, \dots, X_p^{k+1}]$
		\EndFor
	\end{algorithmic}
\end{algorithm}

The following results establish the convergence properties of Gluon-MVR-3 under different parameter settings. We show how to choose parameters to obtain fast convergence rate in subsection \ref{subsec:cp-Gluon-MVR-3}. 

\begin{theorem}\label{th:cs-Gluon-mvr-3}
	Let Assumptions \ref{as:L0L1smooth}, \ref{as:boundedvariance}, \ref{as:rho}, and Assumption \ref{as:HV} hold. Let $X^0, ..., X^{K-1}$ be the iterates of Algorithm \ref{alg:gluon_cs_mvr-3}, and $M_i^0 = g_i^0 = \nabla_i f_{\xi^0} (X^0)$. Denote $\alpha \eqdef 1-\beta$ and $v \eqdef \frac{1-q}{1-\alpha}$. 
	
	1. If $L_i^1 =0$, then for $q < \alpha$, we have 
	\begin{eqnarray*}
		&& \min_{k=0, ..., K-1} \sum_{i=1}^p t_i \E \left[  \|\nabla_i f(X^k) \|_{(i)\star} \right] \\ 
		&\leq&   \frac{\Delta^0}{\eta K} + \frac{2\sum_{i=1}^p t_i \rho \sigma}{\alpha K} +       \frac{4\sqrt{q \alpha} \sum_{i=1}^p t_i  \rho \sigma}{\sqrt{(2-\alpha)(\alpha+\beta q)}}   \\ 
		&& +  \frac{4\sqrt{\alpha} (1-q) \sum_{i=1}^p\delta_i t_i^2 \rho \eta}{\sqrt{(2-\alpha)(\alpha+\beta q)q}}   +  \frac{2\sum_{i=1}^p \delta_i t_i^2 \eta}{\sqrt{\alpha}}  \\ 
		&&   +    \frac{4\sqrt{2\alpha} \sum_{i=1}^p t_i  \rho \sigma}{q K \sqrt{v+\alpha-1}}     + \frac{\sum_{i=1}^p L_i^0 t_i^2 \eta}{2}; 
	\end{eqnarray*}
	for $q\geq \alpha$, 
	\begin{eqnarray*}
		&& \min_{k=0, ..., K-1} \sum_{i=1}^p t_i \E \left[  \|\nabla_i f(X^k) \|_{(i)\star} \right] \\ 
		&\leq&  \frac{\Delta^0}{\eta K} + \frac{8\sum_{i=1}^p t_i \rho \sigma}{\alpha K} +       \frac{4\sqrt{q \alpha} \sum_{i=1}^p t_i  \rho \sigma}{\sqrt{(2-\alpha)(\alpha+\beta q)}}  \\ 
		&&  +  \frac{4\sqrt{\alpha} (1-q) \sum_{i=1}^p\delta_i t_i^2 \rho \eta}{\sqrt{(2-\alpha)(\alpha+\beta q)q}}  \\ 
		&&   +  \frac{2\sum_{i=1}^p \delta_i t_i^2 \eta}{\sqrt{\alpha}} + \frac{\sum_{i=1}^p L_i^0 t_i^2 \eta}{2}. 
	\end{eqnarray*}
	
	2.  If $L_i^1 \neq 0$, we let $\eta \leq \min_i \frac{1}{L_i^1 t_i}$. Then for $q<\alpha$, 
	\begin{eqnarray*}
		&& \min_{k=0, ..., K-1} \sum_{i=1}^p t_i \E \left[  \|\nabla_i f(X^\tau) \|_{(i)\star} \right] \\ 
		&\leq&  \frac{2 \Delta^0}{\eta K} + \frac{4\sum_{i=1}^p t_i \rho \sigma}{\alpha K} +       \frac{8\sqrt{q \alpha} \sum_{i=1}^p t_i  \rho \sigma}{\sqrt{(2-\alpha)(\alpha+\beta q)}}  \\ 
		&&  +  \frac{8\sqrt{\alpha} (1-q) \sum_{i=1}^p\delta_i t_i^2 \rho \eta}{\sqrt{(2-\alpha)(\alpha+\beta q)q}}  +  \frac{4\sum_{i=1}^p \delta_i t_i^2 \eta}{\sqrt{\alpha}}  \\ 
		&&   +    \frac{8\sqrt{2\alpha} \sum_{i=1}^p t_i  \rho \sigma}{q K \sqrt{v+\alpha-1}}    + {\sum_{i=1}^p L_i^0 t_i^2 \eta}; 
	\end{eqnarray*}
	for $q\geq \alpha$, we have 
	\begin{eqnarray*}
		&& \min_{k=0, ..., K-1} \sum_{i=1}^p t_i \E \left[  \|\nabla_i f(X^k) \|_{(i)\star} \right] \\ 
		&\leq&  \frac{2 \Delta^0}{\eta K} + \frac{16\sum_{i=1}^p t_i \rho \sigma}{\alpha K} +       \frac{8\sqrt{q \alpha} \sum_{i=1}^p t_i  \rho \sigma}{\sqrt{(2-\alpha)(\alpha+\beta q)}}  \\ 
		&&  +  \frac{8\sqrt{\alpha} (1-q) \sum_{i=1}^p\delta_i t_i^2 \rho \eta}{\sqrt{(2-\alpha)(\alpha+\beta q)q}}  \\ 
		&&   +   \frac{4\sum_{i=1}^p \delta_i t_i^2 \eta}{\sqrt{\alpha}}  + {\sum_{i=1}^p L_i^0 t_i^2 \eta}. 
	\end{eqnarray*}
	
\end{theorem}

\subsection{Choices of Parameters for Gluon-MVR-3}\label{subsec:cp-Gluon-MVR-3}

In this subsection, we discuss how to achieve fast convergence rate for Algorithm \ref{alg:gluon_cs_mvr-3}. From Theorem \ref{th:cs-Gluon-mvr-3}, for $q\geq \alpha$, we can get 
\begin{eqnarray*}
	&& \min_{k=0, ..., K-1} \sum_{i=1}^p t_i \E \left[  \|\nabla_i f(X^\tau) \|_{(i)\star} \right] \\ 
	&\leq& {\cal O} \left(   \frac{1}{\eta K} + \frac{1}{\alpha K}  + \sqrt{\alpha}  + \frac{\sqrt{\alpha} \eta}{q} + \frac{\eta}{\sqrt{\alpha}} \right). 
\end{eqnarray*}
By choosing $\eta = \alpha = \frac{1}{K^{2/3}}$, and $q \geq \frac{1}{K^{2/3}}$, we can get ${\cal O} \left(  \frac{1}{K^{1/3}}  \right)$ convergence rate. For $q<\alpha$, we can obtain the following upper-bound 
\begin{eqnarray*}
	&& \min_{k=0, ..., K-1} \sum_{i=1}^p t_i \E \left[  \|\nabla_i f(X^\tau) \|_{(i)\star} \right] \\ 
	&\leq& {\cal O} \left(   \frac{1}{\eta K} + \frac{1}{\alpha K}  + \sqrt{q}  + \frac{ \eta}{\sqrt{q}} + \frac{\eta}{\sqrt{\alpha}} + \frac{1}{qK} \right). 
\end{eqnarray*}
By choosing $\eta = q = \frac{1}{K^{2/3}}$ and $\alpha > \frac{1}{K^{2/3}}$, we can achieve ${\cal O} \left(  \frac{1}{K^{1/3}}  \right)$ convergence rate. 

For the case where $L_i^1 \neq 0$, we can also get the above results as long as we choose $\eta$ such that $\eta \leq \min_i \frac{1}{L_i^1 t_i}$ is satisfied. Compared to Gluon-MVR-2, we do not need the restriction on $\eta/\alpha$ for Gluon-MVR-3 in this case, same as Gluon-MVR-1.

\section{EXPERIMENTS}

We evaluate momentum variance reduction (MVR) within the Gluon framework instantiated as Muon and Scion on the FineWeb-10B dataset, training NanoGPT with 124M parameters.
Our implementations build on two open-source repositories:
\textbf{Modded-nanoGPT} \citep{jordan2024modded} and the \textbf{Muon optimizer} \citep{jordan2024muon}. The preformance of Gluon-MVR-3 is worse than Gluon-MVR-2, while better than Gluon-MVR-1. We list the detailed numerical results for Gluon-MVR-3 in the Appendix. 

\paragraph{Parameter Setting}
Unless otherwise noted, we use learning rate $\eta = 3.6 \times 10^{-4}$ and global batch size $B=512$.
For the momentum parameter, the main paper denotes the decay by $\beta\in[0,1)$ and uses $\alpha:=1-\beta$;
where prior code uses an ``$\alpha$'' knob, we report the equivalent $\beta=1-\alpha$ for consistency.
Models are trained for 5{,}000 steps (compute-optimal per Chinchilla scaling).
Additional hyperparameters are in the Appendix. We report the validation loss for one run.

\begin{table}[htbp]
	\footnotesize
	\setlength{\tabcolsep}{5pt}
	\centering
	\caption{Hyperparameter Search Space (Lists Abbreviated with Ellipses)}
	\label{tab:search-space}
	\begin{tabularx}{\columnwidth}{@{}lX@{}}
		\toprule
		$\eta$ & $\{3.6\times10^{-4},\, 4.0\times10^{-4}\,\}$ \\
		$\beta$ & $\{\,0.1,\, 0.2,\, \ldots,\, 0.9\,\}$ \\
		$B$ & $\{\,128,\, 256,\, 512\,\}$ \\
		$q$ (Gluon-MVR-2 only) & $\{\,0.1,\, 0.2,\, \ldots,\, 0.9\,\}$ \\
		\bottomrule
	\end{tabularx}
\end{table}

\subsection{Hyperparameter Search Setup}

We sweep the learning rate, momentum parameter, and batch size for all variants, and additionally sweep the MVR parameter $q$ for \textsc{Gluon-MVR-2}.
Table~\ref{tab:search-space} shows the grid; Table~\ref{tab:selected-compact} lists the selections used in Sections~\ref{sec:gmr1}--\ref{sec:gmr2}.

\begin{table}[htbp]
	\scriptsize
	\centering
	\caption{Selected Hyperparameters for Final Validation Losses (Reported in the Paper’s $\beta$ Notation)}
	\label{tab:selected-compact}
	\resizebox{\columnwidth}{!}{%
		\begin{tabular}{@{}lccc@{}}
			\toprule
			& \textbf{Baseline (Gluon)} & \textbf{Best Gluon-MVR-1} & \textbf{Best Gluon-MVR-2} \\
			\midrule
			$\eta$        & $3.6\times10^{-4}$ & $3.6\times10^{-4}$ & $3.6\times10^{-4}$ \\
			$\beta$       & $0.1$ (=$1{-}\alpha$ with $\alpha{=}0.9$) & $0.6$ ($B{=}128$) / $0.5$ ($B{=}512$) & $0.2$ (=$1{-}\alpha$ with $\alpha{=}0.8$) \\
			$B$           & $512$ & $128,\,512$ & $512$ \\
			$q$           & --    & --          & $0.7$ \\
			\bottomrule
	\end{tabular}}
\end{table}

\subsection{Gluon-MVR-1}
\label{sec:gmr1}

We first evaluate \textsc{Gluon-MVR-1}, trained for 5{,}000 steps on a single A100 GPU.

\paragraph{Large Batch ($B{=}512$)}
The baseline reaches a validation loss of $3.832$.
With \textsc{Gluon-MVR-1} at $\beta{=}0.5$, performance is essentially unchanged at $3.829$ (Fig.~\ref{fig:gmr1_512}).
A brief instability appears around step $\sim\!3200$, but both curves converge to the same value.

\begin{figure}[h]
	\centering
	\includegraphics[width=0.7\linewidth]{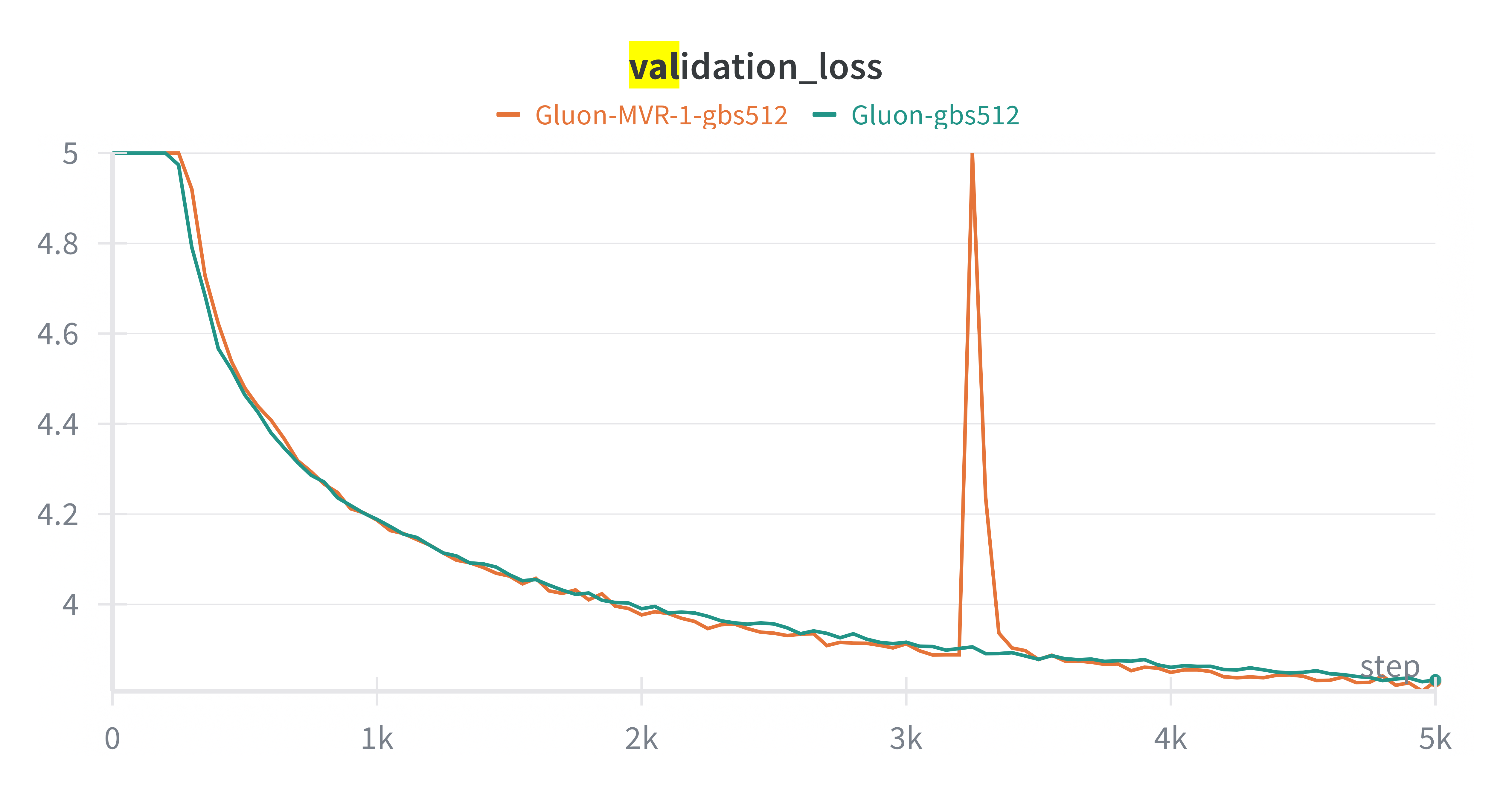}
	\caption{Validation Loss for Gluon (Baseline) VS.\ Gluon-MVR-1 at Global Batch Size $B{=}512$}
	\label{fig:gmr1_512}
\end{figure}

\paragraph{Small Batch ($B{=}128$)}
At smaller batches, \textsc{Gluon-MVR-1} is beneficial: the baseline ($\beta{=}0.1$; code $\alpha{=}0.9$) finishes at $4.256$, while \textsc{Gluon-MVR-1} with $\beta{=}0.6$ improves to $4.107$ (Fig.~\ref{fig:gmr1_128}), $\approx 0.15$ reduction.

\begin{figure}[h]
	\centering
	\includegraphics[width=0.7\linewidth]{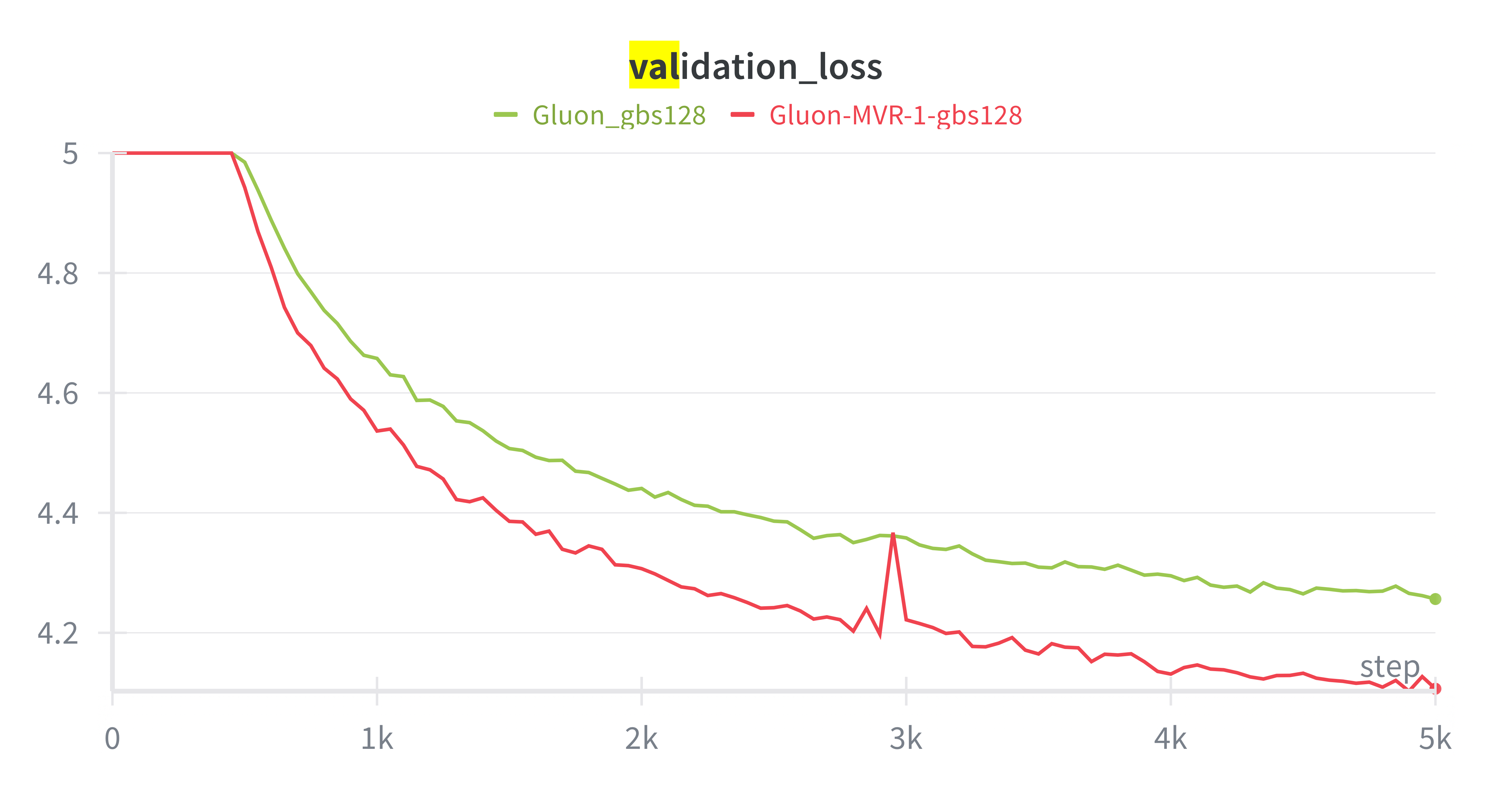}
	\caption{Validation Loss for Gluon (Baseline) VS.\ Gluon-MVR-1 at Global Batch Size $B{=}128$}
	\label{fig:gmr1_128}
\end{figure}

\subsection{Gluon-MVR-2}
\label{sec:gmr2}

We next evaluate \textsc{Gluon-MVR-2}, where MVR is controlled by $q\in(0,1]$.
Training uses 5{,}001 steps on a single A100 GPU.

\paragraph{Results.}
With $B{=}512$, the baseline ($\beta{=}0.1$; code $\alpha{=}0.9$) achieves $3.832$.
Tuning to $\beta{=}0.2$ (code $\alpha{=}0.8$) and $q{=}0.7$ yields $3.682$ (Fig.~\ref{fig:gmr2_512}), an improvement of $\approx 0.15$.
Unlike \textsc{Gluon-MVR-1}, \textsc{Gluon-MVR-2} consistently improves performance in the large-batch regime.

\begin{figure}[h]
	\centering
	\includegraphics[width=0.7\linewidth]{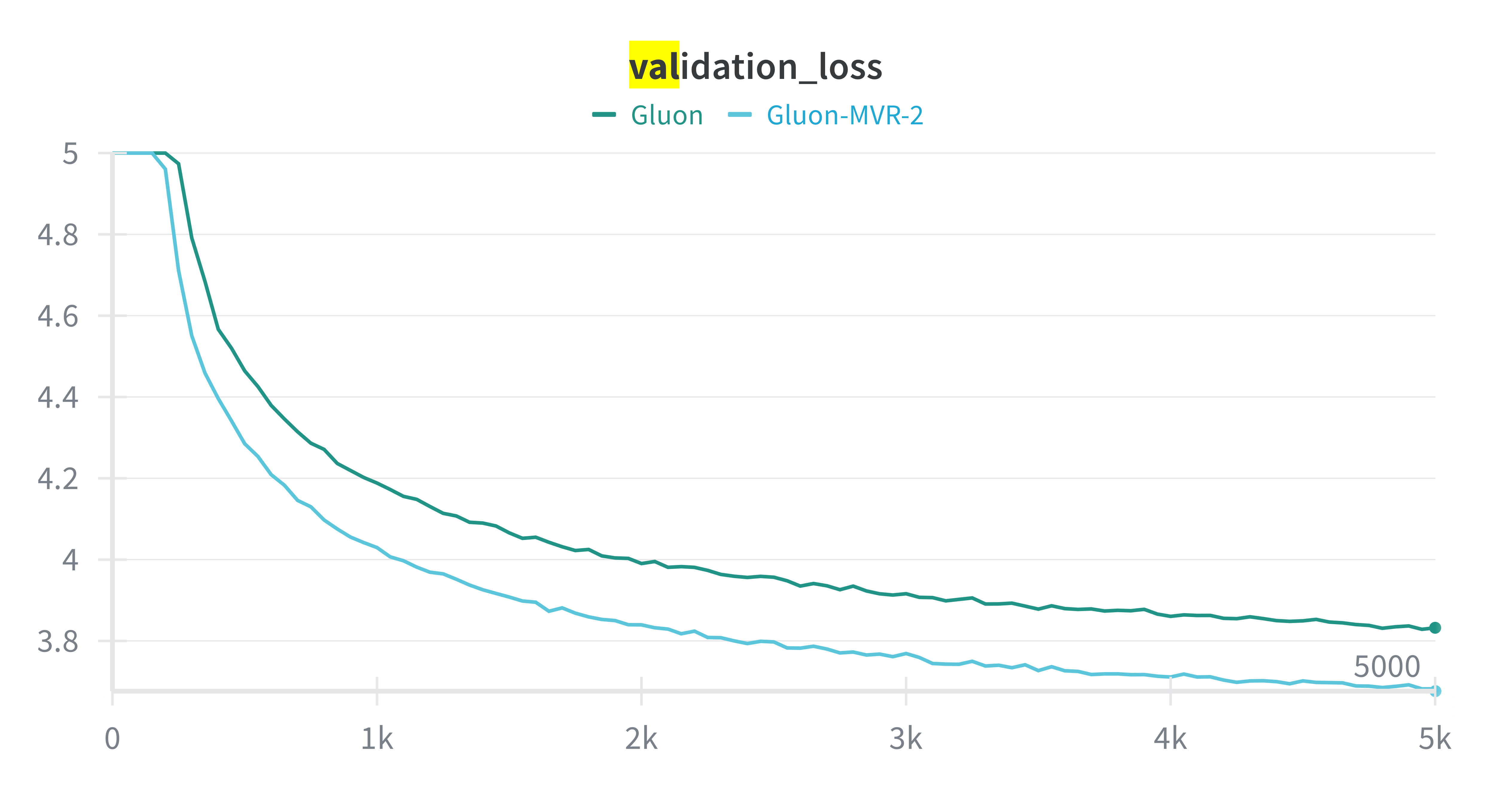}
	\caption{Validation loss for Gluon (baseline) VS.\ Gluon-MVR-2 at global batch size $B{=}512$.}
	\label{fig:gmr2_512}
\end{figure}

\subsection{Conclusion of Experimental Results}

Across our experimental setups, incorporating momentum variance reduction into Gluon (instantiated as Muon/Scion updates) yields meaningful gains.
\textsc{Gluon-MVR-1} is comparable with Gluon at batch size $B{=}512$, while noticeably better than Gluon at batch size $B{=}128$.
\textsc{Gluon-MVR-2}, via its $q$-controlled estimator, improves consistently at batch size $B{=}512$, lowering validation loss by $\sim 0.15$ over the baseline.
These empirical trends align with the theory: MVR reduces variance in the momentum estimate, improving stability and convergence under high-variance training.

\clearpage

\bibliographystyle{plainnat}
\bibliography{gluon_mvr_ref}

\begin{thebibliography}{19}
\providecommand{\natexlab}[1]{#1}
\providecommand{\url}[1]{\texttt{#1}}
\expandafter\ifx\csname urlstyle\endcsname\relax
  \providecommand{\doi}[1]{doi: #1}\else
  \providecommand{\doi}{doi: \begingroup \urlstyle{rm}\Url}\fi

\bibitem[Arjevani et~al.(2023)Arjevani, Carmon, Duchi, Foster, Srebro, and
  Woodworth]{arjevani2023lower}
Yossi Arjevani, Yair Carmon, John~C Duchi, Dylan~J Foster, Nathan Srebro, and
  Blake Woodworth.
\newblock Lower bounds for non-convex stochastic optimization.
\newblock \emph{Mathematical Programming}, 199\penalty0 (1):\penalty0 165--214,
  2023.

\bibitem[Chang et~al.(2025)Chang, Liu, and Yuan]{chang2025convergence}
Da~Chang, Yongxiang Liu, and Ganzhao Yuan.
\newblock On the convergence of muon and beyond.
\newblock \emph{arXiv preprint arXiv:2509.15816}, 2025.

\bibitem[Cutkosky and Orabona(2019)]{cutkosky2019momentum}
Ashok Cutkosky and Francesco Orabona.
\newblock Momentum-based variance reduction in non-convex sgd.
\newblock \emph{Advances in neural information processing systems}, 32, 2019.

\bibitem[Huang et~al.(2025)Huang, Luo, and Chen]{huang2025limuonlightfastmuon}
Feihu Huang, Yuning Luo, and Songcan Chen.
\newblock Limuon: Light and fast muon optimizer for large models, 2025.
\newblock URL \url{https://arxiv.org/abs/2509.14562}.

\bibitem[H{\"u}bler et~al.(2024)H{\"u}bler, Yang, Li, and
  He]{hubler2024parameter}
Florian H{\"u}bler, Junchi Yang, Xiang Li, and Niao He.
\newblock Parameter-agnostic optimization under relaxed smoothness.
\newblock In \emph{International Conference on Artificial Intelligence and
  Statistics}, pages 4861--4869. PMLR, 2024.

\bibitem[Jordan et~al.()Jordan, Jin, Boza, Jiacheng, Cecista, Newhouse, and
  Bernstein]{jordan6muon}
Keller Jordan, Yuchen Jin, Vlado Boza, You Jiacheng, Franz Cecista, Laker
  Newhouse, and Jeremy Bernstein.
\newblock Muon: An optimizer for hidden layers in neural networks, 2024.
\newblock \emph{URL https://kellerjordan. github. io/posts/muon}, 6.

\bibitem[Jordan et~al.(2024)Jordan, Bernstein, Rappazzo, Vlado, Jiacheng,
  Cesista, and Koszarsky]{jordan2024modded}
Keller Jordan, Jeremy Bernstein, Ben Rappazzo, B.~Vlado, Y.~Jiacheng,
  F.~Cesista, and B.~Koszarsky.
\newblock Modded-nanogpt: Speedrunning the nanogpt baseline.
\newblock GitHub repository, 2024.
\newblock URL \url{https://github.com/KellerJordan/modded-nanogpt}.
\newblock Additional contributors: \texttt{@fernbear.bsky.social},
  \texttt{@Grad62304977}.

\bibitem[Kinga et~al.(2015)Kinga, Adam, et~al.]{kinga2015method}
Diederik Kinga, Jimmy~Ba Adam, et~al.
\newblock Adam: A method for stochastic optimization.
\newblock In \emph{International conference on learning representations
  (ICLR)}, volume~5. California;, 2015.

\bibitem[Kovalev(2025)]{kovalev2025understanding}
Dmitry Kovalev.
\newblock Understanding gradient orthogonalization for deep learning via
  non-euclidean trust-region optimization.
\newblock \emph{arXiv preprint arXiv:2503.12645}, 2025.

\bibitem[Li and Hong(2025)]{li2025note}
Jiaxiang Li and Mingyi Hong.
\newblock A note on the convergence of muon and further.
\newblock \emph{arXiv e-prints}, pages arXiv--2502, 2025.

\bibitem[Loshchilov and Hutter(2019)]{loshchilovdecoupled}
Ilya Loshchilov and Frank Hutter.
\newblock Decoupled weight decay regularization.
\newblock In \emph{International Conference on Learning Representations}, 2019.

\bibitem[Pethick et~al.(2025{\natexlab{a}})Pethick, Xie, Antonakopoulos, Zhu,
  Silveti-Falls, and Cevher]{jordan2024muon}
Thomas Pethick, Wanyun Xie, Kimon Antonakopoulos, Zhenyu Zhu, Antonio
  Silveti-Falls, and Volkan Cevher.
\newblock {Scion}.
\newblock GitHub repository, 2025{\natexlab{a}}.
\newblock URL \url{https://github.com/LIONS-EPFL/scion.git}.

\bibitem[Pethick et~al.(2025{\natexlab{b}})Pethick, Xie, Antonakopoulos, Zhu,
  Silveti-Falls, and Cevher]{pethick2025training}
Thomas Pethick, Wanyun Xie, Kimon Antonakopoulos, Zhenyu Zhu, Antonio
  Silveti-Falls, and Volkan Cevher.
\newblock Training deep learning models with norm-constrained lmos.
\newblock \emph{arXiv preprint arXiv:2502.07529}, 2025{\natexlab{b}}.

\bibitem[Reddi et~al.(2019)Reddi, Kale, and Kumar]{reddi2019convergence}
Sashank~J Reddi, Satyen Kale, and Sanjiv Kumar.
\newblock On the convergence of adam and beyond.
\newblock \emph{arXiv preprint arXiv:1904.09237}, 2019.

\bibitem[Riabinin et~al.(2025)Riabinin, Shulgin, Gruntkowska, and
  Richt{\'a}rik]{riabinin2025gluon}
Artem Riabinin, Egor Shulgin, Kaja Gruntkowska, and Peter Richt{\'a}rik.
\newblock Gluon: Making muon \& scion great again!(bridging theory and practice
  of lmo-based optimizers for llms).
\newblock \emph{arXiv preprint arXiv:2505.13416}, 2025.

\bibitem[Sfyraki and Wang(2025)]{sfyraki2025lions}
Maria-Eleni Sfyraki and Jun-Kun Wang.
\newblock Lions and muons: Optimization via stochastic frank-wolfe.
\newblock \emph{arXiv preprint arXiv:2506.04192}, 2025.

\bibitem[Shen et~al.(2025)Shen, Huang, Huang, Shen, and
  Zhang]{shen2025convergence}
Wei Shen, Ruichuan Huang, Minhui Huang, Cong Shen, and Jiawei Zhang.
\newblock On the convergence analysis of muon.
\newblock \emph{arXiv preprint arXiv:2505.23737}, 2025.

\bibitem[Xie et~al.(2024)Xie, Zhou, Li, Lin, and Yan]{xie2024adan}
Xingyu Xie, Pan Zhou, Huan Li, Zhouchen Lin, and Shuicheng Yan.
\newblock Adan: Adaptive nesterov momentum algorithm for faster optimizing deep
  models.
\newblock \emph{IEEE Transactions on Pattern Analysis and Machine
  Intelligence}, 46\penalty0 (12):\penalty0 9508--9520, 2024.

\bibitem[Zhang et~al.(2019)Zhang, He, Sra, and Jadbabaie]{zhang2019gradient}
Jingzhao Zhang, Tianxing He, Suvrit Sra, and Ali Jadbabaie.
\newblock Why gradient clipping accelerates training: A theoretical
  justification for adaptivity.
\newblock \emph{arXiv preprint arXiv:1905.11881}, 2019.

\end{thebibliography}

\clearpage
\appendix
\thispagestyle{empty}

\onecolumn
\aistatstitle{Supplementary Materials}

\tableofcontents

\newpage

\section{Gluon-MVR-3: Empirical Evaluation}

We now empirically evaluate the third instantiation of Gluon with Momentum Variance Reduction (\textsc{Gluon-MVR-3}), corresponding to Algorithm~5 in Section~6. 
This variant augments the standard momentum update by an additional correction term 
$\beta(\nabla_i f_{\xi_k}(X_k) - \nabla_i f_{\xi_k}(X_{k-1}))$, 
which aims to further reduce the stochastic variance of the momentum estimate.

\paragraph{Experimental Setup.}
All experiments are conducted on a single NVIDIA A100 GPU. 
We train a NanoGPT model with $n_{\mathrm{layer}}=12$, $n_{\mathrm{head}}=6$, and $n_{\mathrm{embd}}=768$, corresponding to approximately $124$M parameters. 
The dataset follows the \textsc{FineWeb-10B} configuration. 
We use a global batch size of $B=512$, device batch size $b=32$, and sequence length $L=1024$. 
Each configuration is trained for $K=5{,}000$ iterations.

We perform a narrow grid search over the hyperparameters:
\[
\eta \in \{3.6\times10^{-4}\}, \quad 
\alpha \in \{0.70, 0.8, 0.9, 0.99\}, \quad 
q \in \{0.3, 0.5, 0.7, 0.9\},
\]
resulting in $16$ independent runs.
All experiments are monitored via \textsc{Weights \& Biases}, and the best-performing configuration is selected based on validation loss at convergence.

\paragraph{Results.}
Figure~\ref{fig:gluon_mvr3_validation} reports the validation loss trajectory for the configuration
$\eta = 3.6\times10^{-4}$, $\alpha = 0.8$, and $q = 0.5$.
The baseline \textsc{Gluon} optimizer (\(\eta = 3.6\times 10^{-4},\,\alpha=0.9\)) achieves a final validation loss of $3.832$ after $K=5{,}000$ iterations.
In contrast, \textsc{Gluon-MVR-3} attains a lower final loss of $3.742$, corresponding to an improvement of approximately $0.09$ in the large-batch regime.
A transient spike around iteration $3{,}100$ is observed, which we attribute to a temporary amplification of the MVR accumulator due to correlated gradient samples; the model subsequently recovers rapidly and stabilizes.

\begin{figure}[h]
	\centering
	\includegraphics[width=0.72\linewidth]{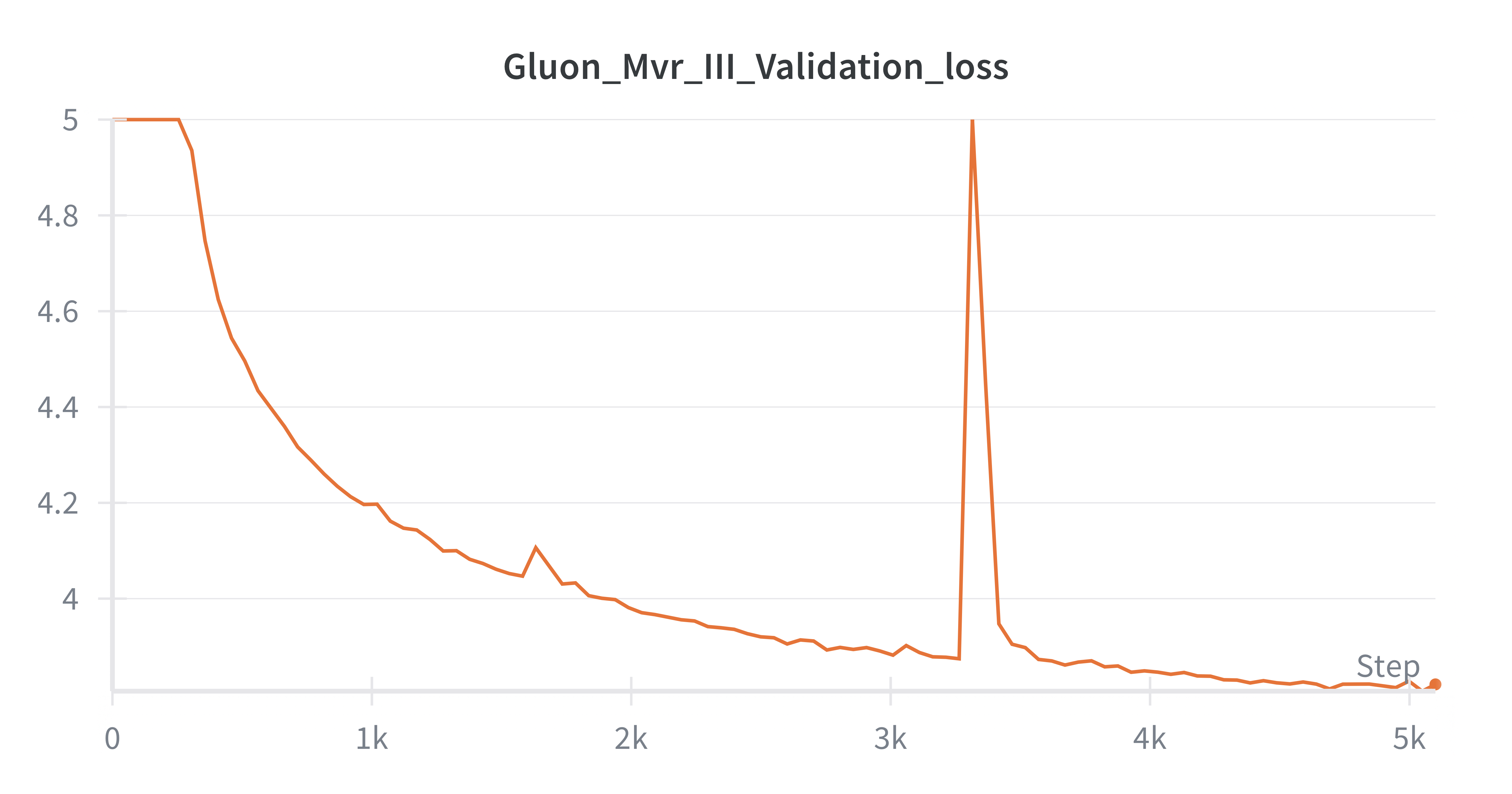}
	\caption{Validation Loss Trajectory of \textsc{Gluon-MVR-3} with $\eta=3.6\times10^{-4}$, $\alpha=0.8$, $q=0.5$, and $B=512$. 
		A Transient Instability Near Step $3{,}100$ is Followed by Stable Convergence to a Lower Steady-state Loss Than the Baseline \textsc{Gluon}.}
	\label{fig:gluon_mvr3_validation}
\end{figure}

\newpage 

\section{GLUON WITH CONSTANT STEP SIZE}

\begin{theorem}\label{th:cs-Gluon}
	Consider the constant step size regime: $\beta^k \equiv \beta$, $t_i^k \equiv t_i \eta$ in Algorithm 1 in \cite{riabinin2025gluon}. Let Assumptions \ref{as:L0L1smooth}, \ref{as:boundedvariance}, and Assumption \ref{as:rho} hold, and fix $\epsilon>0$. Let $X^0, ..., X^{K-1}$ be the iterates of Algorithm 1 in \cite{riabinin2025gluon}, and $M_i^0 = \nabla_i f_{\xi^0} (X^0)$. 
	
	1. If $L_i^1 =0$, then
	\begin{eqnarray*}
		&& \min_{k=0, ..., K-1} \sum_{i=1}^p t_i \E \left[  \|\nabla_i f(X^k) \|_{(i)\star} \right] \\ 
		&\leq&  \frac{\Delta^0}{\eta K} + \frac{2\sum_{i=1}^p t_i \rho \sigma}{\alpha K} + 2\sqrt{\alpha} \sum_{i=1}^p t_i \rho \sigma + \frac{2\sum_{i=1}^p L_i^0t_i^2 \eta}{\alpha} + \frac{\sum_{i=1}^p L_i^0 t_i^2 \eta}{2}. 
	\end{eqnarray*}
	By choosing $\eta = K^{-3/4}$ and $\alpha = K^{-1/2}$, we have $$\min_{k=0, ..., K-1} \sum_{i=1}^p t_i \E \left[  \|\nabla_i f(X^k) \|_{(i)\star} \right]  \leq {\cal O} \left(  \frac{1}{K^{1/4}}  \right).$$ 
	To reach the precision 	$ \min_{k=0, ..., K-1} \sum_{i=1}^p t_i \E \left[  \|\nabla_i f(X^k) \|_{(i)\star} \right]  \leq \epsilon$, it is sufficient to choose the parameters as follows: 
	
	\begin{align}
		& \eta = {\cal O} \left(  \min \left\{  \frac{\epsilon}{\sum_{i=1}^p L_i^0 t_i^2},  \frac{\epsilon^3}{(\sum_{i=1}^p L_i^0t_i^2) (\sum_{i=1}^pt_i)^2 \rho^2 \sigma^2}  \right\}  \right), \\ 
		&  \alpha = {\cal O} \left(  \min \left\{  1, \frac{\epsilon^2}{(\sum_{i=1}^p t_i)^2 \rho^2 \sigma^2}  \right\}  \right), \\
		& K = {\cal O} \left(  \max \left\{  \frac{\sum_{i=1}^p t_i \rho \sigma}{\epsilon}, \frac{(\sum_{i=1}^p t_i)^3 \rho^3 \sigma^3}{\epsilon^3},   \frac{\sum_{i=1}^p L_i^0 t_i^2 \Delta^0}{\epsilon^2}, \frac{(\sum_{i=1}^p L_i^0t_i^2) (\sum_{i=1}^p t_i)^2 \Delta^0 \rho^2 \sigma^2}{\epsilon^4}  \right\}  \right), 
	\end{align}
	where $\Delta^0 \eqdef f(X^0) - \inf_{X \in {\cal S}} f(X)$. 
	
	2. If $L_i^1 \neq 0$, we let $\frac{\eta}{\alpha} \leq \min_{i}  \frac{1}{5L_i^1 t_i}$. Then $\left(  \frac{2}{\alpha} + \frac{1}{2}  \right) L_i^1 t_i \eta \leq \frac{1}{2}$ for all $i$, and 
	\begin{eqnarray*}
		&& \min_{k=0, ..., K-1} \sum_{i=1}^p t_i \E \left[  \|\nabla_i f(X^k) \|_{(i)\star} \right] \\ 
		&\leq&  \frac{2\Delta^0}{\eta K} + \frac{4\sum_{i=1}^p t_i \rho \sigma}{\alpha K} + 4\sqrt{\alpha} \sum_{i=1}^p t_i \rho \sigma + \frac{4\sum_{i=1}^p L_i^0t_i^2 \eta}{\alpha} + {\sum_{i=1}^p L_i^0 t_i^2 \eta}. 
	\end{eqnarray*}
	By choosing $\eta = \min \left(  \min_{i}  \frac{1}{5L_i^1 t_i} K^{-1/2},  K^{-3/4}   \right)$, and $\alpha = K^{-1/2}$, we have $$\min_{k=0, ..., K-1} \sum_{i=1}^p t_i \E \left[  \|\nabla_i f(X^k) \|_{(i)\star} \right]  \leq {\cal O} \left(  \frac{1}{K^{1/4}} + \frac{\max_{i} L_i^1 t_i}{K^{1/2}}  \right).$$
	To reach the precision 	$ \min_{k=0, ..., K-1} \sum_{i=1}^p t_i \E \left[  \|\nabla_i f(X^k) \|_{(i)\star} \right]  \leq \epsilon$, it is sufficient to choose the parameters as follows: 
	
	\begin{align}
		& \eta = {\cal O} \left(  \min \left\{  \frac{\epsilon}{\sum_i L_i^0 t_i^2},  \frac{\epsilon^3}{(\sum_i L_i^0t_i^2) (\sum_i t_i)^2 \rho^2 \sigma^2}, \frac{1}{5\max_iL_i^1 t_i},  \frac{\epsilon^2}{5\max_iL_i^1 t_i (\sum_i t_i)^2 \rho^2 \sigma^2}  \right\}  \right), \\ 
		&  \alpha = {\cal O} \left(  \min \left\{  1, \frac{\epsilon^2}{(\sum_{i=1}^p t_i)^2 \rho^2 \sigma^2}  \right\}  \right), \\
		& K = {\cal O} \left(  \max \left\{  \frac{\sum_{i=1}^p t_i \rho \sigma}{\epsilon}, \frac{(\sum_{i=1}^p t_i)^3 \rho^3 \sigma^3}{\epsilon^3},   \frac{\sum_{i=1}^p L_i^0 t_i^2 \Delta^0}{\epsilon^2}, \frac{\max_i L_i^1 t_i \Delta^0}{\epsilon}, \right. \right.  \nonumber \\ 
		& \quad \quad \quad \quad  \left. 	\left.  \frac{(\sum_{i=1}^p t_i)^2 \Delta^0 \rho^2 \sigma^2}{\max_i L_i^1 t_i \epsilon^3},  \frac{(\sum_{i=1}^p L_i^0t_i^2) (\sum_{i=1}^p t_i)^2 \Delta^0 \rho^2 \sigma^2}{\epsilon^4}  \right\}  \right)
	\end{align}
	
\end{theorem}

\begin{proof}
	Firstly, similar to the proof of (28) in \cite{riabinin2025gluon}, we can obtain 
	\begin{align}
		\sum_{i=1}^p \sum_{k=0}^{K-1} t_i \eta \E \left[  \| \nabla_i f(X^k)\|_{(i)\star}  \right] \leq \Delta^0 + &\sum_{i=1}^p \left[  2 \sum_{k=0}^{K-1} t_i \eta \E \left[  \|M_i^k - \nabla_i f(X^k)\|_{(i)\star}   \right] \right. \label{eq:sumfgrad-cs} \\ 
		& \quad \quad \left.  +  \sum_{k=0}^{K-1} \frac{L_i^0}{2} t_i^2 \eta^2 + \sum_{k=0}^{K-1} \frac{L_i^1 t_i^2 \eta^2}{2} \E \left[  \| \nabla_i f(X^k)\|_{(i)\star}  \right] \right]. \nonumber 
	\end{align}
	
	We introduce the following notation: $\mu_i^k \eqdef M_i^k - \nabla_i f(X^k)$, $\gamma_i^k \eqdef \nabla_i f_{\xi^k} (X^k) - \nabla_i f(X^k)$, $\alpha = 1- \beta$, and $S_i^k \eqdef \nabla_i f(X^{k-1}) - \nabla_i f(X^k)$. Then we have 
	\begin{eqnarray*}
		\mu_i^k &=& M_i^k - \nabla_i f(X^k) \\ 
		&=& \alpha \gamma_i^k + \beta S_i^k + \beta \mu_i^{k-1} \\ 
		&=& \beta^k \mu_i^0 + \sum_{\tau =1}^k \beta^{k-\tau} \alpha \gamma_i^\tau + \sum_{\tau=1}^k \beta^{k+1-\tau} S_i^\tau. 
	\end{eqnarray*}
	
	Hence we can obtain 
	\begin{eqnarray*}
		&& \E \left[  \| M_i^k - \nabla_i f(X^k)\|_{(i)\star}  \right] \\ 
		&=& \E \left[  \|\mu_i^k\|_{(i)\star}  \right] \\ 
		&\overset{(a)}{\leq}& \beta^k \E \left[  \|\mu_i^0\|_{(i)\star}  \right] + \E \left[  \| \sum_{\tau =1}^k \beta^{k-\tau} \alpha \gamma_i^\tau\|_{(i)\star} \right] + \sum_{\tau=1}^k \beta^{k+1-\tau} \E \left[  \|S_i^\tau \|_{(i)\star} \right] \\ 
		&\overset{(b)}{\leq}& \beta^k \rho \E \left[  \|\mu_i^0\|_2  \right] + \rho \E \left[  \|\sum_{\tau=1}^k \beta^{k-\tau} \alpha \gamma_i^\tau \|_2 \right] + \sum_{\tau=1}^k \beta^{k+1-\tau} \left(  L_i^0 + L_i^1 \E \left[  \nabla_i f(X^\tau)\|_{(i)\star}  \right]  \right) t_i \eta \\ 
		&\overset{(c)}{\leq}& \beta^k \rho \sqrt{\E \left[  \|\mu_i^0\|_2^2  \right]} + \rho \sqrt{\E \left[  \sum_{\tau=1}^k\| \beta^{k-\tau} \alpha \gamma_i^\tau \|_2^2 \right]} + \frac{L_i^0t_i\eta}{\alpha} + L_i^1 t_i \eta \sum_{\tau=1}^k \beta^{k+1-\tau} \E \left[  \|\nabla_i f(X^\tau)\|_{(i)\star}  \right] \\ 
		&\overset{(d)}{\leq}& \beta^k \rho \sigma + \alpha \rho \sigma \sqrt{\sum_{\tau=1}^k \beta^{2k-2\tau}} +  \frac{L_i^0t_i\eta}{\alpha} + L_i^1 t_i \eta \sum_{\tau=1}^k \beta^{k+1-\tau} \E \left[  \|\nabla_i f(X^\tau)\|_{(i)\star}  \right] \\ 
		&{\leq}& (1-\alpha)^k \rho\sigma + \sqrt{\alpha}\rho\sigma + \frac{L_i^0t_i\eta}{\alpha} +  L_i^1 t_i \eta \sum_{\tau=1}^k \beta^{k+1-\tau} \E \left[  \|\nabla_i f(X^\tau)\|_{(i)\star}  \right], 
	\end{eqnarray*}
	where (a) uses the triangle inequality, (b) uses Assumptions \ref{as:L0L1smooth} and \ref{as:rho}, (c) uses Jensen’s inequality and the fact that samples $\xi^k \sim {\cal D}$ are i.i.d, (d) uses Assumption \ref{as:boundedvariance}. 
	
	Combining the above inequality with (\ref{eq:sumfgrad-cs}) gives 
	\begin{align*}
		\sum_{i=1}^p \sum_{k=0}^{K-1} t_i \eta \E \left[ \| \nabla_i f(X^k)\|_{(i)\star}  \right] \leq \Delta^0 + \sum_{i=1}^p  & \left[   \sum_{k=0}^{K-1} 2(1-\alpha)^kt_i\eta\rho \sigma + \sum_{k=0}^{K-1} 2\sqrt{\alpha} t_i \eta \rho \sigma + \sum_{k=0}^{K-1}  \frac{2L_i^0 t_i^2 \eta^2}{\alpha}   \right. \\ 
		& \quad \left. + \sum_{k=0}^{K-1} 2L_i^1 t_i^2 \eta^2 \sum_{\tau=1}^k \beta^{k+1-\tau} \E \left[  \|\nabla_i f(X^\tau) \|_{(i)\star}  \right] \right.  \\ 
		& \quad  \left. + \sum_{k=0}^{K-1} \frac{L_i^0 t_i^2 \eta^2}{2} + \sum_{k=0}^{K-1}  \frac{L_i^1 t_i^2 \eta^2}{2} \E \left[  \nabla_i f(X^k)\|_{(i)\star}  \right] \right]. 
	\end{align*}
	Since 
	$$
	\sum_{k=0}^{K-1} \sum_{\tau=1}^k \beta^{k+1-\tau} \E \left[  \|\nabla_i f(X^\tau) \|_{(i)\star}  \right]  = \sum_{\tau=1}^{K-1} \sum_{k=\tau}^{K-1} \beta^{k+1-\tau}  \E \left[  \|\nabla_i f(X^\tau) \|_{(i)\star}  \right]  \leq \frac{1}{\alpha} \sum_{k=0}^{K-1}  \E \left[  \|\nabla_i f(X^\tau) \|_{(i)\star}  \right], 
	$$
	which implies that 
	\begin{align*}
		\sum_{i=1}^p \sum_{k=0}^{K-1} t_i \eta \E \left[ \| \nabla_i f(X^k)\|_{(i)\star}  \right] \leq \Delta^0 + \sum_{i=1}^p & \left[   \frac{2t_i \eta \rho \sigma}{\alpha} + 2K\sqrt{\alpha} t_i \eta \rho \sigma + \frac{2KL_i^0 t_i^2 \eta^2}{\alpha}  + \frac{K L_i^0 t_i^2 \eta^2}{2}   \right. \\ 
		& \quad \left.  +  \sum_{k=0}^{K-1} \left(  \frac{2}{\alpha} + \frac{1}{2}  \right) L_i^1 t_i^2 \eta^2 \E \left[  \|\nabla_i f(X^k) \|_{(i)\star} \right]   \right]. 
	\end{align*}
	
	Now we consider two options: (1) $L_i^1 = 0$ for all $i \in \{  1, ..., p  \}$ and (2) $L_i^1 \neq 0$, for all $i \in \{  1, ..., p  \}$.  
	
	{\bf Case 1:} $L_i^1 = 0$ for all $i \in \{  1, ..., p  \}$. In this case, 
	\begin{eqnarray*}
		&& \min_{k=0, ..., K-1} \sum_{i=1}^p t_i \E \left[  \|\nabla_i f(X^k) \|_{(i)\star} \right] \\ 
		&\leq&  \frac{1}{K} \sum_{k=0}^{K-1} \sum_{i=1}^p t_i \E \left[  \|\nabla_i f(X^k) \|_{(i)\star} \right] \\ 
		&\leq&  \frac{\Delta^0}{\eta K} + \frac{2\sum_{i=1}^p t_i \rho \sigma}{\alpha K} + 2\sqrt{\alpha} \sum_{i=1}^p t_i \rho \sigma + \frac{2\sum_{i=1}^p L_i^0t_i^2 \eta}{\alpha} + \frac{\sum_{i=1}^p L_i^0 t_i^2 \eta}{2}. 
	\end{eqnarray*}

	{\bf Case 2:} $L_i^1 \neq 0$, for all $i \in \{  1, ..., p  \}$. First we let $\frac{\eta}{\alpha} \leq \min_{i}  \frac{1}{5L_i^1 t_i}$. Then $\left(  \frac{2}{\alpha} + \frac{1}{2}  \right) L_i^1 t_i \eta \leq \frac{1}{2}$ for all $i$, and 
	\begin{eqnarray*}
		&& \min_{k=0, ..., K-1} \sum_{i=1}^p t_i \E \left[  \|\nabla_i f(X^k) \|_{(i)\star} \right] \\ 
		&\leq&  \frac{1}{K} \sum_{k=0}^{K-1} \sum_{i=1}^p t_i \E \left[  \|\nabla_i f(X^k) \|_{(i)\star} \right] \\ 
		&\leq&  \frac{2\Delta^0}{\eta K} + \frac{4\sum_{i=1}^p t_i \rho \sigma}{\alpha K} + 4\sqrt{\alpha} \sum_{i=1}^p t_i \rho \sigma + \frac{4\sum_{i=1}^p L_i^0t_i^2 \eta}{\alpha} + {\sum_{i=1}^p L_i^0 t_i^2 \eta}. 
	\end{eqnarray*}
	
\end{proof}

\newpage 

\section{PROOFS FOR MUON-MVR}

\subsection{Two Lemmas}

\begin{lemma}\label{lem:mvr}
	Let the parameters $\eta$ and $\beta$ satisfy the following inequality: $\eta \geq \beta \max\{  \|X^0\|, \|X^*\|  \}$. Then the following inequality holds:
	\begin{equation}
		\E[\| M^k - \nabla f(X^k)\|_\star ] 
		\leq
		(1-\alpha)^k\rho\sigma
		+2\sqrt{\alpha}\rho\sigma
		+\frac{4\rho\eta L}{\sqrt{\alpha}}.
	\end{equation}
\end{lemma}

\begin{proof}

We can express $M^{k+1} - \nabla f(X^{k+1})$ as follows:
\begin{align*}
	M^{k+1} - \nabla f(X^{k+1}) 
	&\overset{\text{(a)}}{=} (1 - \alpha)(M^k - \nabla f_{\xi^{k+1}}(X^k)) + \nabla f_{\xi^{k+1}}(X^{k+1}) - \nabla f(X^{k+1}) \\
	&\overset{\text{(b)}}{=} (1 - \alpha)(M^k - \nabla f(X^k)) + \Delta_{k+1},
\end{align*}
where (a) uses eq.~(4); (b) uses the following definition of $\Delta_{k+1}$:
\begin{equation}\label{eq:Delta}
	\Delta_{k+1} = \nabla f_{\xi^{k+1}}(X^{k+1}) - \nabla f(X^{k+1}) + (1 - \alpha)(\nabla f(X^k) - \nabla f_{\xi^{k+1}}(X^k)). 
\end{equation}
Hence, $M^k - \nabla f(X^k)$ can be expressed as follows:
\begin{equation*}
	M^k - \nabla f(X^k) = (1 - \alpha)^k (M^0 - \nabla f(X^0)) + \sum_{i=1}^k (1 - \alpha)^{k-i} \Delta_i.
\end{equation*}

Using this, we can upper-bound $\mathbb{E}[\|M^k - \nabla f(X^k)\|_*]$ as follows:
\begin{align*}
	&\mathbb{E}[\|M^k - \nabla f(X^k)\|_*] \\
	&\quad \overset{\text{(a)}}{\leq} (1 - \alpha)^k \mathbb{E}[\|M^0 - \nabla f(X^0)\|_*] + \mathbb{E}\big[\big\|\sum\nolimits_{i=1}^k (1 - \alpha)^{k-i} \Delta_i\big\|_*\big] \\
	&\quad \overset{\text{(b)}}{\leq} (1 - \alpha)^k \rho \sqrt{\mathbb{E}[\|M^0 - \nabla f(X^0)\|_2^2]} + \rho \sqrt{\mathbb{E}\big[\big\|\sum\nolimits_{i=1}^k (1 - \alpha)^{k-i} \Delta_i\big\|_2^2\big]} \\
	&\quad \overset{\text{(c)}}{\leq} (1 - \alpha)^k \rho \sigma + \rho \sqrt{\mathbb{E}\big[\big\|\sum\nolimits_{i=1}^k (1 - \alpha)^{k-i} \Delta_i\big\|_2^2\big]} \\
	&\quad \overset{\text{(d)}}{=} (1 - \alpha)^k \rho \sigma + \rho \sqrt{\mathbb{E}\big[\sum\nolimits_{i=1}^k (1 - \alpha)^{2(k-i)} \|\Delta_i\|_2^2\big]} \\
	&\quad \overset{\text{(e)}}{=} (1 - \alpha)^k \rho \sigma + \rho \sqrt{2\mathbb{E}\big[\sum\nolimits_{i=1}^k (1 - \alpha)^{2(k-i)} \alpha^2 \|\nabla f_{\xi^i}(X^i) - \nabla f(X^i)\|_2^2\big]} \\
	&\qquad + \rho \sqrt{2\mathbb{E}\big[\sum\nolimits_{i=1}^k (1 - \alpha)^{2(k-i+1)} \|\nabla f_{\xi^i}(X^i) - \nabla f(X^i) - \nabla f_{\xi^{i}}(X^{i-1}) + \nabla f(X^{i-1})\|_2^2\big]} \\
	&\quad \overset{\text{(f)}}{\leq} (1 - \alpha)^k \rho \sigma + \alpha\rho\sigma \sqrt{2\sum\nolimits_{i=1}^k (1 - \alpha)^{2(k-i)}} \\
	&\qquad + \rho \sqrt{2\mathbb{E}\big[\sum\nolimits_{i=1}^k (1 - \alpha)^{2(k-i+1)} \|\nabla f_{\xi^i}(X^i) - \nabla f(X^i) - \nabla f_{\xi^{i}}(X^{i-1}) + \nabla f(X^{i-1})\|_2^2\big]} \\
	&\quad \overset{\text{(g)}}{\leq} (1 - \alpha)^k \rho \sigma + \alpha\rho\sigma \sqrt{2\sum\nolimits_{i=1}^k (1 - \alpha)^{2(k-i)}} \\
	&\qquad + \rho \sqrt{2\mathbb{E}\big[\sum\nolimits_{i=1}^k (1 - \alpha)^{2(k-i+1)} \|\nabla f_{\xi^i}(X^i) - \nabla f_{\xi^{i}}(X^{i-1})\|_*^2\big]} \\
	&\quad \overset{\text{(h)}}{\leq} (1 - \alpha)^k \rho \sigma + \alpha\rho\sigma \sqrt{2\sum\nolimits_{i=1}^k (1 - \alpha)^{2(k-i)}} + 2\rho\eta L \sqrt{2\sum\nolimits_{i=1}^k (1 - \alpha)^{2(k-i+1)}} \\
	&\quad \leq (1 - \alpha)^k \rho \sigma + 2\sqrt{\alpha} \rho \sigma + \frac{4\rho\eta L}{\sqrt{\alpha}},
\end{align*}
where (a) uses the triangle inequality; (b) uses Assumption \ref{eq:rho} and the Jensen's inequality; (c) and (f) use Assumption \ref{eq:sigma}; (d) uses the fact that $\mathbb{E}[\langle \Delta_i, \Delta_j \rangle] = 0$ for all $i \neq j$; (e) uses the definition of $\Delta_i$ in (\ref{eq:Delta}), the triangle inequality and the inequality $\sqrt{a + b} \leq \sqrt{a} + \sqrt{b}$; (g) uses the inequality $\mathbb{E}[\|\zeta - \mathbb{E}[\zeta]\|_2^2] \leq \mathbb{E}[\|\zeta\|_2^2]$ and Assumption \ref{eq:rho}; (h) uses Assumption \ref{eq:L} and the inequality $\|X^i - X^{i-1}\| \leq 2\eta$, which is implied by Lemma~6 of \citet{kovalev2025understanding}.

\end{proof}

\begin{lemma}\label{lem:tr}
	Let the parameters $\eta$ and $\beta$ satisfy the following inequality: $\eta \geq \beta \max\{  \|X^0\|, \|X^*\|  \}$. Then the  following inequality holds:
	\begin{equation}
		f(X^{k+1}) \leq (1-\beta) f(X^k) + \beta f(X^*) + 4\eta^2 L + 2\eta\|M^k - \nabla f(X^k)\|_\star.
	\end{equation}
\end{lemma}

\begin{proof}
The proof can be found in the proof of Theorem~4 of \citet{kovalev2025understanding}.
\end{proof}

\subsection{Proof of Theorem \ref{th:muon-mvr}}

From Lemma \ref{lem:tr}, we have 
\begin{eqnarray*}
	\E \left[  f(X^{k+1}) - f(X^*)  \right] &\leq& (1-\beta) \E \left[  f(X^k) - f(X^*)  \right]  + 4\eta^2 L + 2\eta \E \left[\|M^k - \nabla f(X^k)\|_\star \right]\\
	&\leq& (1-\beta)^{k+1} \left(  f(X^0) - f(X^*)  \right)  + \sum_{i=0}^k (1-\beta)^i \left( 4\eta^2 L + 2\eta \E \left[\|M^{k-i} - \nabla f(X^k)\|_\star \right]   \right) \\
	&\overset{lemma~\ref{lem:mvr}}{\leq}&  (1-\beta)^{k+1} \left(  f(X^0) - f(X^*)  \right) + \sum_{i=0}^k (1-\beta)^i  \left(  4\eta^2 L  + 4\sqrt{\alpha} \eta \rho \sigma  + \frac{8\rho \eta^2 L}{\sqrt{\alpha}}  \right)  \\ 
	&& +   \sum_{i=0}^k (1-\beta)^i  (1-\alpha)^{k-i} 2\eta \rho \sigma \\ 
	&\leq& (1-\beta)^{k+1} \left(  f(X^0) - f(X^*)  \right) +  \frac{4\eta^2 L}{\beta}  + \frac{4\sqrt{\alpha}\eta \rho \sigma}{\beta}  + \frac{8\rho \eta^2 L}{\sqrt{\alpha} \beta} + (k+1) (1-\beta)^k 2\eta \rho \sigma, 
\end{eqnarray*}
where we use $\alpha \geq \beta$ in the last inequality.

\subsection{Proof of Corollary \ref{co:muon-mvr}}

From Theorem \ref{th:muon-mvr} and the fact that $(1-x)^{\frac{1}{x}} \leq e^{-1}$ for $0<x<1$, we have 
\begin{eqnarray*}
	\E \left[  f(X^K) - f(X^*)  \right] &\leq& (1-\beta)^K \left(  f(X^0) - f(X^*)  \right) +  \frac{4\eta^2 L}{\beta}  + \frac{4\sqrt{\alpha}\eta \rho \sigma}{\beta}  + \frac{8\rho \eta^2 L}{\sqrt{\alpha} \beta} + K (1-\beta)^{K-1} 2\eta \rho \sigma \\ 
	&=&  (1-\beta)^K \left(  f(X^0) - f(X^*)  \right)  + K (1-\beta)^{K-1} 2\beta \rho \sigma D  + 4\beta LD^2  +  4\sqrt{\alpha} \rho \sigma D + 8\sqrt{\alpha} \rho LD^2  \\ 
	&\leq& \frac{f(X^0) - f(X^*)}{K^2}   +  \frac{2\beta \rho \sigma D}{(1-\beta)K} +  4\beta LD^2  +  4\sqrt{\alpha} \rho \sigma D + 8\sqrt{\alpha} \rho LD^2  \\ 
	&=& \frac{f(X^0) - f(X^*)}{K^2}  +  \frac{4\rho \sigma D \ln K}{(1-\beta)K^2}  +  \frac{8LD^2 \ln K}{K}  + \frac{4\sqrt{2}\rho \sigma D \sqrt{\ln K}}{\sqrt{K}}  +  \frac{8\sqrt{2}\rho LD^2\sqrt{\ln K}}{\sqrt{K}}, 
\end{eqnarray*}
where the last inequality comes from 
\begin{eqnarray*}
(1 - \beta)^K &=& \left(  1 - \frac{2\ln K}{K}  \right)^{\frac{K}{2\ln K} \cdot 2\ln K} \\ 
&\leq& e^{-2\ln K} \\ 
&=& \frac{1}{K^2}.  
\end{eqnarray*}

\newpage 

\section{PROOFS FOR GLUON-MVR-1}

\subsection{Proof of Theorem \ref{th:cs-Gluon-mvr}} 

First  we also define $\mu_i^k \eqdef M_i^k - \nabla_i f(X^k)$, $\gamma_i^k \eqdef \nabla_i f_{\xi^k} (X^k) - \nabla_i f(X^k)$, and $\alpha = 1- \beta$. Then we have 
\begin{eqnarray*}
	\mu_i^k &=& M_i^k - \nabla_i f(X^k) \\
	&=& \beta \mu_i^{k-1} + \alpha \gamma_i^k + \beta Z_i^k \\ 
	&=& \beta^k \mu_i^0 + \sum_{\tau =1}^k \beta^{k-\tau} \alpha \gamma_i^\tau + \sum_{\tau=1}^k \beta^{k+1-\tau} Z_i^\tau, 
\end{eqnarray*}
where we denote $Z_i^k \eqdef \nabla_i f_{\xi^k} (X^k) - \nabla_i f_{\xi^k} (X^{k-1}) - (\nabla_i f(X^k) - \nabla_i f(X^{k-1}))$. 

We upper-bound $\E \left[  \| \sum_{\tau=1}^k \beta^{k+1-\tau} Z_i^\tau\|_{(i)\star}  \right]$ as follows. 

\begin{eqnarray}
	\E \left[  \| \sum_{\tau=1}^k \beta^{k+1-\tau} Z_i^\tau\|_{(i)\star}  \right] &\leq& \rho \E \left[  \| \sum_{\tau=1}^k \beta^{k+1-\tau} Z_i^\tau\|_2  \right] \nonumber \\ 
	&\leq& \rho \sqrt{\E \left[   \| \sum_{\tau=1}^k \beta^{k+1-\tau} Z_i^\tau\|_2^2  \right]}  \nonumber \\ 
	&\leq& \rho \sqrt{\sum_{\tau=1}^k \beta^{2k+2-2\tau} \E \left[  \|Z_i^\tau \|_2^2 \right]}  \nonumber \\ 
	&\leq& \rho \delta_i t_i \eta \sqrt{\sum_{\tau=1}^k \beta^{2k+2-2\tau}}  \nonumber  \\ 
	&\leq& \frac{\rho \delta_i t_i \eta}{\sqrt{\alpha}}, \label{eq:sumZitau}
\end{eqnarray}
where we use Assumption \ref{as:rho} in the first inequality, Jensen's inequality in the second inequality, and the fact that samples $\xi^k \sim {\cal D}$ are i.i.d in the third inequality.  

Then similar to the proof of Theorem \ref{th:cs-Gluon}, we can obtain 
\begin{align*}
	\sum_{i=1}^p \sum_{k=0}^{K-1} t_i \eta \E \left[ \| \nabla_i f(X^k)\|_{(i)\star}  \right] \leq \Delta^0 + \sum_{i=1}^p & \left[   \frac{2t_i \eta \rho \sigma}{\alpha} + 2K\sqrt{\alpha} t_i \eta \rho \sigma + \frac{2K\rho \delta_i t_i^2 \eta^2}{\sqrt{\alpha}}  + \frac{K L_i^0 t_i^2 \eta^2}{2}   \right. \\ 
	& \quad \left.  +  \sum_{k=0}^{K-1}  \frac{1}{2}L_i^1 t_i^2 \eta^2 \E \left[  \|\nabla_i f(X^k) \|_{(i)\star} \right]   \right]. 
\end{align*}

Now we consider two options: (1) $L_i^1 = 0$ for all $i \in \{  1, ..., p  \}$ and (2) $L_i^1 \neq 0$, for all $i \in \{  1, ..., p  \}$.  

{\bf Case 1:} $L_i^1 = 0$ for all $i \in \{  1, ..., p  \}$. In this case, 
\begin{eqnarray*}
	&& \min_{k=0, ..., K-1} \sum_{i=1}^p t_i \E \left[  \|\nabla_i f(X^k) \|_{(i)\star} \right] \\ 
	&\leq&  \frac{1}{K} \sum_{k=0}^{K-1} \sum_{i=1}^p t_i \E \left[  \|\nabla_i f(X^k) \|_{(i)\star} \right] \\ 
	&\leq&  \frac{\Delta^0}{\eta K} + \frac{2\sum_{i=1}^p t_i \rho \sigma}{\alpha K} + 2\sqrt{\alpha} \sum_{i=1}^p t_i \rho \sigma + \frac{2\sum_{i=1}^p \delta_i t_i^2 \rho  \eta}{\sqrt{\alpha}} + \frac{\sum_{i=1}^p L_i^0 t_i^2 \eta}{2}. 
\end{eqnarray*}

{\bf Case 2:} $L_i^1 \neq 0$, for all $i \in \{  1, ..., p  \}$. First we let $\eta \leq \min_i \frac{1}{L_i^1 t_i}$. Then $\frac{1}{2} L_i^1 t_i \eta \leq \frac{1}{2}$ for all $i$, and 
\begin{eqnarray*}
	&& \min_{k=0, ..., K-1} \sum_{i=1}^p t_i \E \left[  \|\nabla_i f(X^k) \|_{(i)\star} \right] \\ 
	&\leq&  \frac{1}{K} \sum_{k=0}^{K-1} \sum_{i=1}^p t_i \E \left[  \|\nabla_i f(X^k) \|_{(i)\star} \right] \\ 
	&\leq&  \frac{2\Delta^0}{\eta K} + \frac{4\sum_{i=1}^p t_i \rho \sigma}{\alpha K} + 4\sqrt{\alpha} \sum_{i=1}^p t_i \rho \sigma + \frac{4\sum_{i=1}^p \delta_i t_i^2 \rho \eta}{\sqrt{\alpha}} + {\sum_{i=1}^p L_i^0 t_i^2 \eta}. 
\end{eqnarray*}

\newpage 

\section{PROOFS FOR GLUON-MVR-1 WITH DECREASING STEP SIZE}

\subsection{Proof of Lemma \ref{lm:lm1ds}}

First we have 
\begin{eqnarray*}
	&& \sum_{k=0}^{K-1} (k+1)^{-\frac{2}{3}}  \sqrt{ \sum_{\tau =0}^k (\beta^{(\tau+1):k} \alpha^\tau)^2 }\\ 
	&=&  \sum_{k=0}^{K-1} (k+1)^{-\frac{2}{3}}  \sqrt{  \sum_{\tau=0}^k (\alpha^\tau)^2 \prod_{s = \tau+1}^k (\beta^s)^2  } \\ 
	&\leq&  \sum_{k=0}^{K-1} (k+1)^{-\frac{2}{3}}  \sqrt{  2 \prod_{s=2}^k (\beta^s)^2 +  \sum_{\tau=2}^k (\alpha^\tau)^2 \prod_{s = \tau+1}^k (\beta^s)^2   } \\ 
	&\leq& \sum_{k=0}^{K-1} (k+1)^{-\frac{2}{3}} \left(  \sqrt{2} \prod_{s=1}^k \beta^s  + \sqrt{  \sum_{\tau=2}^k (\alpha^\tau)^2 \prod_{s = \tau+1}^k (\beta^s)^2   }  \right). 
\end{eqnarray*}
Next we upper-bound the two terms in the last inequalities respectively. 

\begin{eqnarray*}
	\sum_{k=0}^{K-1} (k+1)^{-\frac{2}{3}} \sqrt{2} \prod_{s=1}^k \beta^s  &=& \sqrt{2} \sum_{k=1}^K k^{-\frac{2}{3}} \prod_{s=2}^{k-1} (1 - (s+1)^{-\frac{2}{3}}) \\ 
	&=& \sqrt{2} \sum_{k=1}^K k^{-\frac{2}{3}} \prod_{s=3}^{k} (1 - s^{-\frac{2}{3}}) \\ 
	&\leq& 2\sqrt{2} + \sqrt{2}  \sum_{k=3}^K k^{-\frac{2}{3}} \prod_{s=3}^{k} (1 - s^{-\frac{2}{3}}) \\ 
	&\leq& 2\sqrt{2} +2\sqrt{2}  = 4\sqrt{2}, 
\end{eqnarray*}
where in the last inequality we use Lemma 10(ii) in \cite{hubler2024parameter} with $a=3$, $b=K$, and $p=q=\frac{2}{3}$. 

\begin{eqnarray*}
	\sum_{k=0}^{K-1} (k+1)^{-\frac{2}{3}}  \sqrt{  \sum_{\tau=2}^k (\alpha^\tau)^2 \prod_{s = \tau+1}^k (\beta^s)^2   }  &\leq& \sum_{k=0}^{K-1} (k+1)^{-\frac{2}{3}}  \sqrt{  \sum_{\tau=2}^k (\tau+1)^{-\frac{4}{3}}  \prod_{s = \tau+1}^k (1 - (s+1)^{-\frac{2}{3}})   } \\ 
	&=&  \sum_{k=0}^{K-1} (k+1)^{-\frac{2}{3}}  \sqrt{  \sum_{\tau=3}^{k+1} \tau^{-\frac{4}{3}}  \prod_{s = \tau+1}^{k+1} (1 - s^{-\frac{2}{3}})   } \\ 
	&\leq&  \sum_{k=0}^{K-1} (k+1)^{-\frac{2}{3}}  \sqrt{2e^3 (k+2)^{-\frac{2}{3}}} \\ 
	&\leq& \sqrt{2e^3} \sum_{k=1}^K k^{-1} \\ 
	&\leq& \sqrt{2e^3} (\log K + 1), 
\end{eqnarray*}
where in the second inequality we use Lemma 10(iii) in \cite{hubler2024parameter} with $a=3$, $b=k+1$, $p=\frac{4}{3}$, and $q = \frac{2}{3}$. 

Thus we arrive at 
\begin{eqnarray*}
	\sum_{k=0}^{K-1} (k+1)^{-\frac{2}{3}}  \sqrt{ \sum_{\tau =0}^k (\beta^{(\tau+1):k} \alpha^\tau)^2 } &\leq& 4\sqrt{2} +  \sqrt{2e^3} (\log K + 1) \\ 
	&\leq& 12 + \sqrt{2e^3} \log K. 
\end{eqnarray*}

\subsection{Proof of Lemma \ref{lm:lm2ds}}

Noticing that $t_i^k = t_i \alpha^k$, we have 
\begin{eqnarray*}
	\sum_{k=0}^{K-1} (k+1)^{-\frac{2}{3}}  \sqrt{ \sum_{\tau =1}^k (\beta^{(\tau+1):k} t_i^\tau)^2 } &=& \sum_{k=0}^{K-1} (k+1)^{-\frac{2}{3}}  \sqrt{ \sum_{\tau =1}^k (\beta^{(\tau+1):k} t_i \alpha^\tau)^2 } \\ 
	&\leq& t_i \sum_{k=0}^{K-1} (k+1)^{-\frac{2}{3}}  \sqrt{ \sum_{\tau =0}^k (\beta^{(\tau+1):k}  \alpha^\tau)^2 } \\ 
	&\leq& t_i \left(  12 + \sqrt{2e^3} \log K  \right), 
\end{eqnarray*}
where we use Lemma \ref{lm:lm1ds} in the last inequality.

\subsection{Proof of Theorem \ref{th:ds-Gluon-mvr}}

For Algorithm \ref{alg:ds_gluon_mvr}, we have 
\begin{eqnarray*}
	\mu_i^k &=& M_i^k - \nabla_i f(X^k) \\
	&=& \beta^k \mu_i^{k-1} + \alpha^k \gamma_i^k + \beta^k Z_i^k \\ 
	&=& \beta^{1:k} \mu_i^0 + \sum_{\tau =1}^k \beta^{(\tau+1):k} \alpha^\tau \gamma_i^\tau + \sum_{\tau=1}^k \beta^{\tau:k} Z_i^\tau \\ 
	&=& \sum_{\tau =0}^k \beta^{(\tau+1):k} \alpha^\tau \gamma_i^\tau + \sum_{\tau=1}^k \beta^{\tau:k} Z_i^\tau, 
\end{eqnarray*}
where we use $M_i^0 = \nabla_i f_{\xi^0}(X^0)$ and $\beta^0=0$ in the last equality. Therefore, 
\begin{eqnarray*}
	\E \left[  \|M_i^k - \nabla_i f(X^k)\|_{(i) \star}  \right] &=& \E \left[  \|\mu_i^k\|_{(i)\star}  \right] \\ 
	&\overset{(a)}{\leq}& \E \left[  \| \sum_{\tau =0}^k \beta^{(\tau+1):k} \alpha^\tau \gamma_i^\tau \|_{(i)\star}  \right] + \E \left[  \| \sum_{\tau=1}^k \beta^{\tau:k} Z_i^\tau \|_{(i)\star}  \right] \\ 
	&\overset{(b)}{\leq}& \rho  \E \left[  \| \sum_{\tau =0}^k \beta^{(\tau+1):k} \alpha^\tau \gamma_i^\tau \|_2  \right] + \rho \E \left[  \| \sum_{\tau=1}^k \beta^{\tau:k} Z_i^\tau \|_2  \right] \\ 
	&\overset{(c)}{\leq}& \rho \sqrt{ \E \left[  \| \sum_{\tau =0}^k \beta^{(\tau+1):k} \alpha^\tau \gamma_i^\tau \|_2^2  \right] } + \rho \sqrt{\E \left[  \| \sum_{\tau=1}^k \beta^{\tau:k} Z_i^\tau \|_2^2  \right] } \\ 
	&\overset{(d)}{\leq}& \rho \sqrt{ \sum_{\tau =0}^k (\beta^{(\tau+1):k} \alpha^\tau)^2 \|\gamma_i^\tau\|_2^2 } + \rho \sqrt{  \sum_{\tau=1}^k (\beta^{\tau:k})^2 \|Z_i^\tau\|_2^2  } \\ 
	&\overset{(d)}{\leq}& \rho \sigma  \sqrt{ \sum_{\tau =0}^k (\beta^{(\tau+1):k} \alpha^\tau)^2 } + \rho \delta_i  \sqrt{  \sum_{\tau=1}^k (\beta^{\tau:k} t_i^\tau)^2  }, 
\end{eqnarray*}
where (a) uses the triangle inequality, (b) uses Assumption \ref{as:rho}, (c) used Jensen's inequality, (d) uses the fact that samples $\xi^k \sim {\cal D}$ are i.i.d, and (d) uses Assimptions \ref{as:boundedvariance} and \ref{as:HV}. 

Combining the last inequality, Lemmas \ref{lm:lm1ds}, \ref{lm:lm2ds}, and (\ref{eq:sumfgrad-cs}), we arrive at 
\begin{align*}
	\sum_{i=1}^p \sum_{k=0}^{K-1} t_i^k \E \left[  \|\nabla_i f(X^k)\|_{(i)\star}  \right]  & \leq  \Delta^0 +  \sum_{i=1}^p  \left[  \rho \sigma t_i (24 + 2\sqrt{2e^3} \log K) + \rho \delta_i t_i^2 (24 + 2\sqrt{2e^3} \log K) \right. \\ 
	& \quad \left.  +  \sum_{k=0}^{K-1} \frac{L_i^0}{2} (t_i^k)^2  + \sum_{k=0}^{K-1} \frac{L_i^1}{2} \E \left[  \|\nabla_i f(X^k)\|_{(i)\star} (t_i^k)^2 \right]     \right] \\ 
	&\leq \Delta^0 +  \sum_{i=1}^p  \left[  \rho \sigma t_i (24 + 2\sqrt{2e^3} \log K) + \rho \delta_i t_i^2 (24 + 2\sqrt{2e^3} \log K) \right. \\ 
	& \quad \left.  +  2L_i^0 t_i^2  + \sum_{k=0}^{K-1} \frac{L_i^1}{2} \E \left[  \|\nabla_i f(X^k)\|_{(i)\star} (t_i^k)^2 \right]     \right], 
\end{align*}
where in the last inequality we use the following estimation. 

\begin{eqnarray*}
	\sum_{k=0}^{K-1} (t_i^k)^2 &=& t_i^2 \sum_{k=0}^{K-1} (k+1)^{-\frac{4}{2}} \\ 
	&\leq& t_i^2 \left(  1 + \int_{1}^{+\infty} \frac{1}{z^{4/3}} dz  \right) \\ 
	&\leq& 4t_i^2. 
\end{eqnarray*}

Now we consider two options: (1) $L_i^1 = 0$ for all $i \in \{  1, ..., p  \}$ and (2) $L_i^1 \neq 0$, for all $i \in \{  1, ..., p  \}$.  

{\bf Case 1:} $L_i^1 = 0$ for all $i \in \{  1, ..., p  \}$. In this case, 

\begin{align*}
	& \sum_{i=1}^p \sum_{k=0}^{K-1} t_i^k \E \left[  \|\nabla_i f(X^k)\|_{(i)\star}  \right] \\ 
	&\leq \Delta^0 +  \sum_{i=1}^p  \left[  \rho \sigma t_i (24 + 2\sqrt{2e^3} \log K) + \rho \delta_i t_i^2 (24 + 2\sqrt{2e^3} \log K) + 2L_i^0 t_i^2  \right], 
\end{align*}
and therefore, 
\begin{eqnarray*}
	&&	\min_{k=0, ..., K-1} \sum_{i=1}^p t_i \E \left[  \|\nabla_i f(X^k) \|_{(i)\star}  \right] \\ 
	&\leq& \frac{1}{K} \sum_{k=0}^{K-1} \sum_{i=1}^p t_i \E \left[  \|\nabla_i f(X^k) \|_{(i)\star}  \right] \\ 
	&\leq& \frac{1}{K^{1/3}} \sum_{k=0}^{K-1} \sum_{i=1}^p t_i (1+k)^{-2/3} \E \left[  \|\nabla_i f(X^k) \|_{(i)\star}  \right] \\ 
	&=&  \frac{1}{K^{1/3}} \sum_{k=0}^{K-1} \sum_{i=1}^p t_i^k \E \left[  \|\nabla_i f(X^k) \|_{(i)\star}  \right] \\ 
	&\leq& \frac{\Delta^0}{K^{1/3}} + \frac{1}{K^{1/3}}  \sum_{i=1}^p  \left[  \rho \sigma t_i (24 + 2\sqrt{2e^3} \log K) + \rho \delta_i t_i^2 (24 + 2\sqrt{2e^3} \log K) + 2L_i^0 t_i^2  \right]. 
\end{eqnarray*}

{\bf Case 2:} $L_i^1 \neq 0$ for all $i \in \{  1, ..., p  \}$. Let us choose $t_i = \frac{1}{L_i^1}$. Then 

\begin{align*}
	& \sum_{i=1}^p \sum_{k=0}^{K-1} t_i^k \E \left[  \|\nabla_i f(X^k)\|_{(i)\star}  \right] \\ 
	&\leq 2\Delta^0 +  \sum_{i=1}^p  \left[  \rho \sigma t_i (48 + 4\sqrt{2e^3} \log K) + \rho \delta_i t_i^2 (48 + 4\sqrt{2e^3} \log K) + 4L_i^0 t_i^2  \right], 
\end{align*}
and hence 
\begin{eqnarray*}
	&&	\min_{k=0, ..., K-1} \sum_{i=1}^p t_i \E \left[  \|\nabla_i f(X^k) \|_{(i)\star}  \right] \\ 
	&\leq& \frac{1}{K} \sum_{k=0}^{K-1} \sum_{i=1}^p t_i \E \left[  \|\nabla_i f(X^k) \|_{(i)\star}  \right] \\ 
	&\leq& \frac{1}{K^{1/3}} \sum_{k=0}^{K-1} \sum_{i=1}^p t_i (1+k)^{-2/3} \E \left[  \|\nabla_i f(X^k) \|_{(i)\star}  \right] \\ 
	&=&  \frac{1}{K^{1/3}} \sum_{k=0}^{K-1} \sum_{i=1}^p t_i^k \E \left[  \|\nabla_i f(X^k) \|_{(i)\star}  \right] \\ 
	&\leq& \frac{\Delta^0}{K^{1/3}} + \frac{1}{K^{1/3}}  \sum_{i=1}^p  \left[  \frac{\rho \sigma}{L_i^1} (24 + 2\sqrt{2e^3} \log K) + \frac{\rho \delta_i}{(L_i^1)^2}  (24 + 2\sqrt{2e^3} \log K) + \frac{2L_i^0}{(L_i^1)^2}  \right]. 
\end{eqnarray*}

\newpage 

\section{PROOFS FOR GLUON-MVR-2}

\subsection{Two Technical Lemmas}

\begin{lemma}\label{lm:sumvs-j}
	Denote $\alpha = 1-\beta$ and $v = \frac{1-q}{1-\alpha}$. Assume $\alpha <1$. Then we have 
	$$
	\sum_{j=1}^k \alpha^2 \beta^{2k-2j}  + 2 \sum_{j<s} \alpha^2 \beta^{2k-2j} v^{s-j}  \leq \frac{2\alpha}{(2-\alpha)(\alpha+\beta q)}. 
	$$
\end{lemma}

\begin{proof}
	
	First we have
	$$
	\sum_{j=1}^k \alpha^2 \beta^{2k-2j}  + 2 \sum_{j<s} \alpha^2 \beta^{2k-2j} v^{s-j} \leq \frac{\alpha^2}{1-\beta^2} + 2 \sum_{j<s} \alpha^2 \beta^{2k-2j} v^{s-j}. 
	$$
	Next we consider the cases where $v=1$ and $v\neq 1$ respectively. For $v=1$, we have 
	\begin{eqnarray*}
		\frac{\alpha^2}{1-\beta^2}  + 2\sum_{j<s} \alpha^2 \beta^{2k-2j} v^{s-j} &=&  \frac{\alpha^2}{1-\beta^2}  + 2\sum_{j<s} \alpha^2 \beta^{2k-2j}  \\ 
		&=&   \frac{\alpha^2}{1-\beta^2}  + 2 \sum_{j=1}^k \sum_{s=j+1}^k \alpha^2 \beta^{2k-2j}  \\ 
		&=&   \frac{\alpha^2}{1-\beta^2}  + 2 \sum_{j=1}^k  \alpha^2 \beta^{2k-2j} (k-j) \\ 
		&\leq&   \frac{\alpha^2}{1-\beta^2}  + 2\alpha^2 \frac{\beta^2}{(1-\beta^2)^2} \\ 
		&\leq& \frac{2}{(2-\alpha)^2}. 
	\end{eqnarray*}
	For $v \neq 1$, we have 
	\begin{eqnarray*}
		&& \frac{\alpha^2}{1-\beta^2}  + 2\sum_{j<s} \alpha^2 \beta^{2k-2j} v^{s-j} \\ 
		&=&  \frac{\alpha^2}{1-\beta^2}  + 2\sum_{j=1}^k \sum_{s=j+1}^k \alpha^2 \beta^{2k-2j} v^{s-j}  \\ 
		&=& \frac{\alpha^2}{1-\beta^2} + 2 \sum_{j=1}^k \alpha^2 \beta^{2k-2j} \frac{v(1-v^{k-j})}{1-v} \\ 
		&=& \frac{\alpha^2}{1-\beta^2} + \frac{2\alpha^2v}{1-v} \left(  \sum_{j=1}^k \beta^{2k-2j} - \sum_{j=1}^k (\beta^2 v)^{k-j}  \right) \\ 
		&=& \frac{\alpha^2}{1-\beta^2} + \frac{2\alpha^2v}{1-v} \left(  \frac{1-\beta^2k}{1-\beta^2}  -  \frac{1-(\beta^2v)^k}{1-\beta^2v}  \right) \\ 
		&=& \frac{\alpha^2}{1-\beta^2} + \frac{2\alpha^2v}{1-v} \left(  \frac{1}{1-\beta^2} - \frac{1}{1-\beta^2v}  \right) - \frac{2\alpha^2v}{1-v} \left(  \frac{\beta^2k}{1-\beta^2} - \frac{1 - (\beta^2v)^k}{1-\beta^2v}  \right) \\ 
		&=&  \frac{\alpha^2}{1-\beta^2} + \frac{2\alpha^2v\beta^2}{(1-\beta^2)(1-\beta^2v)} - \frac{2\alpha^2v\beta^{2k}(1-v^k)}{(1-v)(1-\beta^2v)} - \frac{2\alpha^2 v \beta^{2k+2}}{(1-\beta^2)(1-\beta^2v)}\\ 
		&\leq& \frac{\alpha^2}{1-\beta^2} + \frac{2\alpha^2v\beta^2}{(1-\beta^2)(1-\beta^2v)} \\ 
		&=& \frac{\alpha(1 + \beta(1-q))}{(2-\alpha) (1 - \beta^2v)}\\ 
		&\leq& \frac{2\alpha}{(2-\alpha)(\alpha+\beta q)}, 
	\end{eqnarray*}
	where in the first inequality we use $(1-v^k)/(1-v)>0$ and $1-\beta^2v = 1-\beta(1-q)>0$, 
	
	When $v=1$, we have $q=\alpha$, which implies that $\frac{2\alpha}{(2-\alpha)(\alpha+\beta q)} = \frac{2}{(2-\alpha)^2}$. Hence for all $v$, we actually have 
	$$
	\frac{\alpha^2}{1-\beta^2}  + 2\sum_{j<s} \alpha^2 \beta^{2k-2j} v^{s-j}  \leq \frac{2\alpha}{(2-\alpha)(\alpha+\beta q)}. 
	$$
	
\end{proof}

\begin{lemma}\label{lm:sum1-qj}
	Denote $\alpha = 1-\beta$ and $v = \frac{1-q}{1-\alpha}$. Assume $q<\alpha<1$. Then we have 
	$$
	\sum_{j=1}^k \alpha^2 \beta^{2k-2j} (1-q)^j  +  2\sum_{j<s} \alpha^2 \beta^{2k-2j} v^{s-j} (1-q)^j \leq \frac{2\alpha (1-q)^{k+1}}{v+\alpha-1}. 
	$$
\end{lemma}

\begin{proof}
	
	When $q<\alpha<1$, we know $v>1$. Then we have 
	\begin{eqnarray*}
		&& \sum_{j=1}^k \alpha^2 \beta^{2k-2j} (1-q)^j  + 2 \sum_{j<s} \alpha^2 \beta^{2k-2j} v^{s-j} (1-q)^j  \\ 
		&=& \sum_{j=1}^k \alpha^2 \beta^{2k-j} v^j  + 2 \sum_{j=1}^k \sum_{s=j+1}^k \alpha^2  \beta^{2k-j} v^{s} \\ 
		&=&  \sum_{j=1}^k \alpha^2 \beta^{2k-j} v^j  + 2 \sum_{j=1}^k \alpha^2 \beta^{2k-j} \frac{v^{j+1} - v^{k+1}}{1-v} \\ 
		&\leq&  2 \sum_{j=0}^k \alpha^2 \beta^{2k-j} \frac{v^j - v^{k+1}}{1-v} \\ 
		&=& \frac{2\beta^{2k} \alpha^2 }{1-v} \frac{1 - (\frac{v}{\beta})^{k+1}}{1 - \frac{v}{\beta}}  -  \frac{2\alpha^2 \beta^k v^{k+1}}{1-v} \frac{1-\beta^{k+1}}{\alpha}  \\ 
		&=& \frac{2\alpha^2 \beta^k v}{v-1} \left(   \frac{v^k - \beta(1-q)^k}{\alpha}  - \frac{\beta^{k+1}/v - v^{k}}{1-\alpha-v}   \right) \\ 
		&\leq&  \frac{2\alpha^2 \beta^k v}{v-1} \left(    \frac{v^k - \beta^{k+1}}{\alpha}  - \frac{v^{k} - \beta^{k+1}/v}{v+\alpha-1}   \right) \\ 
		&=&   \frac{2\alpha^2 \beta^k v}{v-1} \left(  \frac{v^k}{\alpha}  - \frac{v^{k+1}}{v+\alpha-1}  + \frac{(1-v)(1+\alpha/v)\beta^{k+1}}{\alpha (v+\alpha-1)}  \right) \\ 
		&\leq& \frac{2\alpha^2 \beta^{k+1} v^{k+1}}{\alpha (v+\alpha-1)} \\ 
		&=& \frac{2\alpha (1-q)^{k+1}}{v+\alpha-1}. 
	\end{eqnarray*}
	
\end{proof}

\subsection{Proof of Theorem \ref{th:cs-Gluon-mvr-2}}

We introduce the following notations: $\mu_i^k \eqdef M_i^k - \nabla_i f(X^k)$, $\gamma_i^k \eqdef g_i^k- \nabla_i f(X^k)$, $\alpha = 1- \beta$, and $S_i^k \eqdef \nabla_i f(X^{k-1}) - \nabla_i f(X^k)$. Then from the update rule of $M_i^k$, we can get 
\begin{eqnarray*}
	\mu_i^k &=& M_i^k - \nabla_i f(X^k) \\ 
	&=& \beta M_i^{k-1} + \alpha g_i^k -  \nabla_i f(X^k) \\ 
	&=& \alpha \gamma_i^k + \beta S_i^k + \beta \mu_i^{k-1} \\ 
	&=& \beta^k \mu_i^0 + \sum_{\tau =1}^k \beta^{k-\tau} \alpha \gamma_i^\tau + \sum_{\tau=1}^k \beta^{k+1-\tau} S_i^\tau. 
\end{eqnarray*}

Similar to the analysis of $\E \left[  \|M_i^k - \nabla_i f(X^k) \|_{(i)\star}  \right]$ in the proof of Theorem \ref{th:cs-Gluon}, we can obtain 
\begin{eqnarray}
	\E \left[  \|M_i^k - \nabla_i f(X^k) \|_{(i)\star}  \right] &\leq& (1-\alpha)^k \rho\sigma +    \rho \sqrt{\E \left[  \| \sum_{\tau=1}^k \beta^{k-\tau} \alpha \gamma_i^\tau \|_2^2 \right]}  \nonumber   \\ 
	&&  + \frac{L_i^0t_i\eta}{\alpha} +  L_i^1 t_i \eta \sum_{\tau=1}^k \beta^{k+1-\tau} \E \left[  \|\nabla_i f(X^\tau)\|_{(i)\star}  \right], \label{eq:mik-pre-gluon-mvr-2}
\end{eqnarray}

Next we estimate $\E \left[  \| \sum_{\tau=1}^k \beta^{k-\tau} \alpha \gamma_i^\tau \|_2^2 \right]$. For $k\geq 1$, 
$$
g_i^k - \nabla_i f(X^k) = q ( \nabla_i f_{\xi^k} (X^k) - \nabla_i f(X^k)) + (1-q) (g_i^{k-1} - \nabla_i f(X^{k-1}) + Z_i^k), 
$$
where $Z_i^k \eqdef \nabla_i f_{\xi^k} (X^k) - \nabla_i f_{\xi^k} (X^{k-1}) - (\nabla_i f(X^k) - \nabla_i f(X^{k-1}))$. Thus for $j<s$, we have 
\begin{eqnarray*}
	&& \E \left[  \left\langle g_i^j - \nabla_i f(X^j),  g_i^s - \nabla_i f(X^s) \right\rangle \right] \\ 
	&=& \E \left[  \E \left[  \left\langle g_i^j - \nabla_i f(X^j),  g_i^s - \nabla_i f(X^s) \right\rangle \right]      |  \{g_i^l\}_{l\leq s-1}, \{X^l\}_{l\leq s}   \right] \\ 
	&=& (1-q) \E \left[  \left\langle g_i^j - \nabla_i f(X^j),  g_i^{s-1} - \nabla_i f(X^{s-1}) \right\rangle \right] \\ 
	&=& (1-q)^{s-j} \E \left[  \|g_i^j - \nabla_i f(X^j)\|_2^2  \right], 
\end{eqnarray*}
and 
\begin{eqnarray*}
	&& \E \left[  \|g_i^k - \nabla_i f(X^k)\|_2^2  \right] \\ 
	&=& (1-q)^2 \E \left[  \|g_i^{k-1} - \nabla_i f(X^{k-1})\|_2^2  \right]  +  \E \left[  \| q ( \nabla_i f_{\xi^k} (X^k) - \nabla_i f(X^k))   + (1-q)Z_i^k \|_2^2 \right] \\ 
	&\leq& (1-q)^2 \E \left[  \|g_i^{k-1} - \nabla_i f(X^{k-1})\|_2^2  \right]  + 2q^2\E \left[  \|  \nabla_i f_{\xi^k} (X^k) - \nabla_i f(X^k)  \right] +  2(1-q)^2 \E \left[  \| Z_i^k\|_2^2  \right] \\ 
	&\leq&  (1-q)^2 \E \left[  \|g_i^{k-1} - \nabla_i f(X^{k-1})\|_2^2  \right]  + 2q^2\sigma^2 + 2(1-q)^2\delta_i^2 \E \left[  \|X^k - X^{k-1}\|_{(i)}^2  \right] \\ 
	&\leq&  (1-q)^2 \E \left[  \|g_i^{k-1} - \nabla_i f(X^{k-1})\|_2^2  \right]  + 2q^2\sigma^2 + 2(1-q)^2\delta_i^2 t_i^2 \eta^2 \\ 
	&\leq& (1-q)^{2k} \E \left[  \|g_i^0 - \nabla_i f(X^0)\|_2^2  \right]  +  \sum_{j=0}^{k-1} (1-q)^{2j} \left(  2q^2\sigma^2 + 2(1-q)^2\delta_i^2 t_i^2 \eta^2  \right) \\ 
	&\leq& (1-q)^k \sigma^2 + 2q\sigma^2 + \frac{2(1-q)^2 \delta_i^2 t_i^2 \eta^2}{q}, 
\end{eqnarray*}
where the first inequality comes from the Jensen's inequality, the second inequality is from Assumptions \ref{as:boundedvariance} and \ref{as:HV}, the third inequality comes from the update rule \ref{eq:Gluon-mvr-2-update}. 

Then we have 
\begin{eqnarray*}
	&& \E \left[  \sum_{j=1}^k \sum_{s=1}^k  \left\langle \alpha \beta^{k-j} \gamma_i^j, \alpha \beta^{k-s} \gamma_i^s\right \rangle   \right] \\ 
	&=& \E \left[  \sum_{j=1}^k \alpha^2 \beta^{2k-2j} \| \gamma_i^j\|_2^2  \right]  +  2 \E \left[  \sum_{j<s}  \left\langle \alpha \beta^{k-j} \gamma_i^j, \alpha \beta^{k-s}\gamma_i^s \right \rangle     \right] \\ 
	&=& \E \left[  \sum_{j=1}^k \alpha^2 \beta^{2k-2j} \| g_i^j - \nabla_i f(X^j)\|_2^2  \right]   +  2 \E \left[  \sum_{j<s} \alpha^2 \beta^{2k-2j} v^{s-j} \|g_i^j - \nabla_i f(X^j)\|_2^2  \right] \\ 
	&\leq& \left(   \sum_{j=1}^k \alpha^2 \beta^{2k-2j}  + 2 \sum_{j<s} \alpha^2 \beta^{2k-2j} v^{s-j}   \right) \left(  2q\sigma^2 + \frac{2(1-q)^2 \delta_i^2 t_i^2 \eta^2}{q}  \right)  \\ 
	&& +  \sum_{j=1}^k \alpha^2 \beta^{2k-2j} (1-q)^j \sigma^2 +  2 \sum_{j<s} \alpha^2 \beta^{2k-2j} v^{s-j} (1-q)^j \sigma^2 \\ 
	&\overset{Lemma \ref{lm:sumvs-j}}{\leq}&  \frac{2\alpha}{(2-\alpha)(\alpha+\beta q)} \left(  2q\sigma^2 + \frac{2(1-q)^2 \delta_i^2 t_i^2 \eta^2}{q}  \right)  +  \sum_{j=1}^k \alpha^2 \beta^{2k-2j} (1-q)^j \sigma^2 \\ 
	&& +  2 \sum_{j<s} \alpha^2 \beta^{2k-2j} v^{s-j} (1-q)^j \sigma^2. 
\end{eqnarray*}
Next we estimate the last two terms in the above inequality. When $q \geq \alpha$, we have 
\begin{eqnarray*}
	&&  \sum_{j=1}^k \alpha^2 \beta^{2k-2j} (1-q)^j \sigma^2 +  2 \sum_{j<s} \alpha^2 \beta^{2k-2j} v^{s-j} (1-q)^j \sigma^2 \\ 
	&\leq&   \sum_{j=1}^k \alpha^2 \beta^{2k-2j} (1-\alpha)^j \sigma^2 +  2 \sum_{j<s} \alpha^2 \beta^{2k-2j} (1-\alpha)^j \sigma^2 \\ 
	&=&   \sum_{j=1}^k \alpha^2 \beta^{2k-j} \sigma^2 + 2 \sum_{j=1}^k \sum_{s=j+1}^k \alpha^2 \beta^{2k-j} \sigma^2  \\ 
	&\leq& 2 \sum_{j=1}^k \alpha^2 \beta^{2k-j} (k+1-j) \sigma^2 \\ 
	&\leq& 2 \beta^k \sigma^2. 
\end{eqnarray*}
It should be noticed that we can improve the above estimation by refined analysis, but the upper-bound of $\E \left[  \|M_i^k - \nabla_i f(X^k) \|_{(i)\star}  \right]$ in (\ref{eq:mik-pre-gluon-mvr-2}) could not be improved within some constant since (\ref{eq:mik-pre-gluon-mvr-2}) contains the term $\beta^k \rho \sigma$. 
When $q<\alpha$, we have $v>1$ and from Lemma \ref{lm:sum1-qj} we have 
$$
\sum_{j=1}^k \alpha^2 \beta^{2k-2j} (1-q)^j \sigma^2 +  2 \sum_{j<s} \alpha^2 \beta^{2k-2j} v^{s-j} (1-q)^j \sigma^2 \leq \frac{2\alpha (1-q)^{k+1} \sigma^2}{v+\alpha-1}. 
$$

Hence for $q\geq \alpha$, we have 
$$
\E \left[  \| \sum_{\tau=1}^k \beta^{k-\tau} \alpha \gamma_i^\tau \|_2^2 \right] \leq  \frac{2\alpha}{(2-\alpha)(\alpha+\beta q)} \left(  2q\sigma^2 + \frac{2(1-q)^2 \delta_i^2 t_i^2 \eta^2}{q}  \right)  +   2 \beta^k \sigma^2, 
$$
which along with (\ref{eq:mik-pre-gluon-mvr-2}) yields 
\begin{eqnarray}
	\E \left[  \|M_i^k - \nabla_i f(X^k) \|_{(i)\star}  \right] &\leq& \beta^k \rho\sigma +   \frac{2\sqrt{q \alpha}\rho \sigma}{\sqrt{(2-\alpha)(\alpha+\beta q)}}   +  \frac{2\sqrt{\alpha} (1-q)\rho \delta_i t_i \eta}{\sqrt{(2-\alpha)(\alpha+\beta q)q}}  +  \sqrt{2}\beta^{k/2} \rho \sigma  \nonumber   \\ 
	&&  + \frac{L_i^0t_i\eta}{\alpha} +  L_i^1 t_i \eta \sum_{\tau=1}^k \beta^{k+1-\tau} \E \left[  \|\nabla_i f(X^\tau)\|_{(i)\star}  \right], \label{eq:mik-gluon-mvr-2-q>}
\end{eqnarray}
and for $q<\alpha$, we have 
$$
\E \left[  \| \sum_{\tau=1}^k \beta^{k-\tau} \alpha \gamma_i^\tau \|_2^2 \right] \leq  \frac{2\alpha}{(2-\alpha)(\alpha+\beta q)} \left(  2q\sigma^2 + \frac{2(1-q)^2 \delta_i^2 t_i^2 \eta^2}{q}  \right)  +  \frac{2\alpha (1-q)^{k+1} \sigma^2}{v+\alpha-1}, 
$$
which along with (\ref{eq:mik-pre-gluon-mvr-2}) implies 
\begin{eqnarray}
	\E \left[  \|M_i^k - \nabla_i f(X^k) \|_{(i)\star}  \right] &\leq& \beta^k \rho\sigma +   \frac{2\sqrt{q \alpha}\rho \sigma}{\sqrt{(2-\alpha)(\alpha+\beta q)}}   +  \frac{2\sqrt{\alpha} (1-q)\rho \delta_i t_i \eta}{\sqrt{(2-\alpha)(\alpha+\beta q)q}}  +  \frac{\sqrt{2\alpha} (1-q)^{\frac{k+1}{2}}\rho \sigma}{\sqrt{v+\alpha-1}}  \nonumber   \\ 
	&&  + \frac{L_i^0t_i\eta}{\alpha} +  L_i^1 t_i \eta \sum_{\tau=1}^k \beta^{k+1-\tau} \E \left[  \|\nabla_i f(X^\tau)\|_{(i)\star}  \right], \label{eq:mik-gluon-mvr-2-q<}
\end{eqnarray}

Next we consider the case where $q\geq \alpha$. From (\ref{eq:sumfgrad-cs}) and (\ref{eq:mik-gluon-mvr-2-q>}), similar to the proof of Theorem \ref{th:cs-Gluon}, we can get 
\begin{align*}
	\sum_{i=1}^p \sum_{k=0}^{K-1} t_i \eta \E \left[ \|  \nabla_i f(X^k)\|_{(i)\star}  \right] \leq \Delta^0 + \sum_{i=1}^p & \left[   \frac{8t_i \eta \rho \sigma}{\alpha} +     \frac{4K\sqrt{q \alpha} t_i \eta \rho \sigma}{\sqrt{(2-\alpha)(\alpha+\beta q)}}   +  \frac{4K\sqrt{\alpha} (1-q)\rho \delta_i t_i^2 \eta^2}{\sqrt{(2-\alpha)(\alpha+\beta q)q}}  \right.  \\ 
	& \quad  \left.   + \frac{2KL_i^0 t_i^2 \eta^2}{\alpha}  + \frac{K L_i^0 t_i^2 \eta^2}{2}   \right. \\ 
	& \quad \left.  +   \sum_{k=0}^{K-1} \left(  \frac{2}{\alpha} + \frac{1}{2}  \right) L_i^1 t_i^2 \eta^2 \E \left[  \|\nabla_i f(X^k) \|_{(i)\star} \right]   \right]. 
\end{align*}

Now we consider two options: (1) $L_i^1 = 0$ for all $i \in \{  1, ..., p  \}$ and (2) $L_i^1 \neq 0$, for all $i \in \{  1, ..., p  \}$.  

{\bf Case 1:} $L_i^1 = 0$ for all $i \in \{  1, ..., p  \}$. In this case, 
\begin{eqnarray*}
	&& \min_{k=0, ..., K-1} \sum_{i=1}^p t_i \E \left[  \|\nabla_i f(X^k) \|_{(i)\star} \right] \\ 
	&\leq&  \frac{1}{K} \sum_{k=0}^{K-1} \sum_{i=1}^p t_i \E \left[  \|\nabla_i f(X^k) \|_{(i)\star} \right] \\ 
	&\leq&  \frac{\Delta^0}{\eta K} + \frac{8\sum_{i=1}^p t_i \rho \sigma}{\alpha K} +       \frac{4\sqrt{q \alpha} \sum_{i=1}^p t_i  \rho \sigma}{\sqrt{(2-\alpha)(\alpha+\beta q)}}   +  \frac{4\sqrt{\alpha} (1-q) \sum_{i=1}^p\delta_i t_i^2 \rho \eta}{\sqrt{(2-\alpha)(\alpha+\beta q)q}}    + \frac{2\sum_{i=1}^p L_i^0t_i^2 \eta}{\alpha} + \frac{\sum_{i=1}^p L_i^0 t_i^2 \eta}{2}. 
\end{eqnarray*}

{\bf Case 2:} $L_i^1 \neq 0$, for all $i \in \{  1, ..., p  \}$. First we let $\frac{\eta}{\alpha} \leq \min_{i}  \frac{1}{5L_i^1 t_i}$. Then $\left(  \frac{2}{\alpha} + \frac{1}{2}  \right) L_i^1 t_i \eta \leq \frac{1}{2}$ for all $i$, and 
\begin{eqnarray*}
	&& \min_{k=0, ..., K-1} \sum_{i=1}^p t_i \E \left[  \|\nabla_i f(X^k) \|_{(i)\star} \right] \\ 
	&\leq&  \frac{1}{K} \sum_{k=0}^{K-1} \sum_{i=1}^p t_i \E \left[  \|\nabla_i f(X^k) \|_{(i)\star} \right] \\ 
	&\leq& \frac{2 \Delta^0}{\eta K} + \frac{16\sum_{i=1}^p t_i \rho \sigma}{\alpha K} +       \frac{8\sqrt{q \alpha} \sum_{i=1}^p t_i  \rho \sigma}{\sqrt{(2-\alpha)(\alpha+\beta q)}}   +  \frac{8\sqrt{\alpha} (1-q) \sum_{i=1}^p\delta_i t_i^2 \rho \eta}{\sqrt{(2-\alpha)(\alpha+\beta q)q}}    + \frac{4\sum_{i=1}^p L_i^0t_i^2 \eta}{\alpha} + {\sum_{i=1}^p L_i^0 t_i^2 \eta}. 
\end{eqnarray*}

Next we consider the case where $q<\alpha$, From (\ref{eq:sumfgrad-cs}) and (\ref{eq:mik-gluon-mvr-2-q<}), we can obtain the following estimation similarly, 
\begin{align*}
	\sum_{i=1}^p \sum_{k=0}^{K-1} t_i \eta \E \left[ \|  \nabla_i f(X^k)\|_{(i)\star}  \right] \leq \Delta^0 + \sum_{i=1}^p & \left[   \frac{2t_i \eta \rho \sigma}{\alpha} +     \frac{4K\sqrt{q \alpha} t_i \eta \rho \sigma}{\sqrt{(2-\alpha)(\alpha+\beta q)}}   +  \frac{4K\sqrt{\alpha} (1-q)\rho \delta_i t_i^2 \eta^2}{\sqrt{(2-\alpha)(\alpha+\beta q)q}}  \right.  \\ 
	& \quad  \left.   +   \frac{4\sqrt{2\alpha}t_i \eta \rho \sigma}{q \sqrt{v+\alpha-1}}     + \frac{2KL_i^0 t_i^2 \eta^2}{\alpha}  + \frac{K L_i^0 t_i^2 \eta^2}{2}   \right. \\ 
	& \quad \left.  +  \sum_{k=0}^{K-1} \left(  \frac{2}{\alpha} + \frac{1}{2}  \right) L_i^1 t_i^2 \eta^2 \E \left[  \|\nabla_i f(X^k) \|_{(i)\star} \right]   \right]. 
\end{align*}

Now we consider two options: (1) $L_i^1 = 0$ for all $i \in \{  1, ..., p  \}$ and (2) $L_i^1 \neq 0$, for all $i \in \{  1, ..., p  \}$.  

{\bf Case 1:} $L_i^1 = 0$ for all $i \in \{  1, ..., p  \}$. In this case, 
\begin{eqnarray*}
	&& \min_{k=0, ..., K-1} \sum_{i=1}^p t_i \E \left[  \|\nabla_i f(X^k) \|_{(i)\star} \right] \\ 
	&\leq&  \frac{1}{K} \sum_{k=0}^{K-1} \sum_{i=1}^p t_i \E \left[  \|\nabla_i f(X^k) \|_{(i)\star} \right] \\ 
	&\leq&  \frac{\Delta^0}{\eta K} + \frac{2\sum_{i=1}^p t_i \rho \sigma}{\alpha K} +       \frac{4\sqrt{q \alpha} \sum_{i=1}^p t_i  \rho \sigma}{\sqrt{(2-\alpha)(\alpha+\beta q)}}   +  \frac{4\sqrt{\alpha} (1-q) \sum_{i=1}^p\delta_i t_i^2 \rho \eta}{\sqrt{(2-\alpha)(\alpha+\beta q)q}}  \\ 
	&&   +    \frac{4\sqrt{2\alpha} \sum_{i=1}^p t_i  \rho \sigma}{q K \sqrt{v+\alpha-1}}      + \frac{2\sum_{i=1}^p L_i^0t_i^2 \eta}{\alpha} + \frac{\sum_{i=1}^p L_i^0 t_i^2 \eta}{2}. 
\end{eqnarray*}

{\bf Case 2:} $L_i^1 \neq 0$, for all $i \in \{  1, ..., p  \}$. First we let $\frac{\eta}{\alpha} \leq \min_{i}  \frac{1}{5L_i^1 t_i}$. Then $\left(  \frac{2}{\alpha} + \frac{1}{2}  \right) L_i^1 t_i \eta \leq \frac{1}{2}$ for all $i$, and 
\begin{eqnarray*}
	&& \min_{k=0, ..., K-1} \sum_{i=1}^p t_i \E \left[  \|\nabla_i f(X^k) \|_{(i)\star} \right] \\ 
	&\leq&  \frac{1}{K} \sum_{k=0}^{K-1} \sum_{i=1}^p t_i \E \left[  \|\nabla_i f(X^k) \|_{(i)\star} \right] \\ 
	&\leq&\frac{2 \Delta^0}{\eta K} + \frac{4\sum_{i=1}^p t_i \rho \sigma}{\alpha K} +       \frac{8\sqrt{q \alpha} \sum_{i=1}^p t_i  \rho \sigma}{\sqrt{(2-\alpha)(\alpha+\beta q)}}   +  \frac{8\sqrt{\alpha} (1-q) \sum_{i=1}^p\delta_i t_i^2 \rho \eta}{\sqrt{(2-\alpha)(\alpha+\beta q)q}}  \\ 
	&&   +    \frac{8\sqrt{2\alpha} \sum_{i=1}^p t_i  \rho \sigma}{q K \sqrt{v+\alpha-1}}   + \frac{4\sum_{i=1}^p L_i^0t_i^2 \eta}{\alpha} + {\sum_{i=1}^p L_i^0 t_i^2 \eta}. 
\end{eqnarray*}

\newpage 

\section{PROOFS FOR GLUON-MVR-3}

\subsection{Poof of Theorem \ref{th:cs-Gluon-mvr-3}}

First we define the following notations same as those in the proof of Theorem \ref{th:cs-Gluon-mvr-2}: $\mu_i^k \eqdef M_i^k - \nabla_i f(X^k)$, $\gamma_i^k \eqdef g_i^k- \nabla_i f(X^k)$, $\alpha = 1- \beta$, and $Z_i^k \eqdef \nabla_i f_{\xi^k} (X^k) - \nabla_i f_{\xi^k} (X^{k-1}) - (\nabla_i f(X^k) - \nabla_i f(X^{k-1}))$. Then we have 
\begin{eqnarray*}
	\mu_i^k &=& M_i^k - \nabla_i f(X^k) \\
	&=& \beta \mu_i^{k-1} + \alpha \gamma_i^k + \beta Z_i^k \\ 
	&=& \beta^k \mu_i^0 + \sum_{\tau =1}^k \beta^{k-\tau} \alpha \gamma_i^\tau + \sum_{\tau=1}^k \beta^{k+1-\tau} Z_i^\tau. 
\end{eqnarray*}

From (\ref{eq:sumZitau}), and similar to the proof of Theorem \ref{th:cs-Gluon-mvr-2}, we can obtain that, for $q\geq \alpha$, 
\begin{align*}
	\sum_{i=1}^p \sum_{k=0}^{K-1} t_i \eta \E \left[ \|  \nabla_i f(X^k)\|_{(i)\star}  \right] \leq \Delta^0 + \sum_{i=1}^p & \left[   \frac{8t_i \eta \rho \sigma}{\alpha} +     \frac{4K\sqrt{q \alpha} t_i \eta \rho \sigma}{\sqrt{(2-\alpha)(\alpha+\beta q)}}   +  \frac{4K\sqrt{\alpha} (1-q)\rho \delta_i t_i^2 \eta^2}{\sqrt{(2-\alpha)(\alpha+\beta q)q}}  \right.  \\ 
	& \quad  \left.   +  \frac{2K \rho \delta_i t_i^2 \eta^2}{\sqrt{\alpha}} + \frac{K L_i^0 t_i^2 \eta^2}{2}   \right. \\ 
	& \quad \left.   +  \sum_{k=0}^{K-1}   \frac{1}{2}  L_i^1 t_i^2 \eta^2 \E \left[  \|\nabla_i f(X^k) \|_{(i)\star} \right]   \right]; 
\end{align*}
and for $q<\alpha$, 
\begin{align*}
	\sum_{i=1}^p \sum_{k=0}^{K-1} t_i \eta \E \left[ \|  \nabla_i f(X^k)\|_{(i)\star}  \right] \leq \Delta^0 + \sum_{i=1}^p & \left[   \frac{2t_i \eta \rho \sigma}{\alpha} +     \frac{4K\sqrt{q \alpha} t_i \eta \rho \sigma}{\sqrt{(2-\alpha)(\alpha+\beta q)}}   +  \frac{4K\sqrt{\alpha} (1-q)\rho \delta_i t_i^2 \eta^2}{\sqrt{(2-\alpha)(\alpha+\beta q)q}}  \right.  \\ 
	& \quad  \left.   +   \frac{4\sqrt{2\alpha}t_i \eta \rho \sigma}{q \sqrt{v+\alpha-1}}     +  \frac{2K\rho \delta_i t_i^2 \eta^2}{\sqrt{\alpha}}  + \frac{K L_i^0 t_i^2 \eta^2}{2}   \right. \\ 
	& \quad \left.  +  \sum_{k=0}^{K-1}  \frac{1}{2}  L_i^1 t_i^2 \eta^2 \E \left[  \|\nabla_i f(X^k) \|_{(i)\star} \right]   \right]. 
\end{align*}
Then we can get the results.

\end{document}